\documentclass{article}
\usepackage{econometrics}

\usepackage{arxiv}
\usepackage[utf8]{inputenc} 
\usepackage[T1]{fontenc}    
\usepackage[table,xcdraw, dvipsnames]{xcolor}
\usepackage[colorlinks=true, linkcolor=blue, urlcolor=blue, citecolor=blue]{hyperref}
\usepackage{url}            
\usepackage{booktabs}       
\usepackage{amsfonts, amsmath, amssymb, bm}       
\usepackage{nicefrac}       
\usepackage{microtype}      
\usepackage{lipsum}		
\usepackage{graphicx}
\usepackage{natbib}
\usepackage{doi}
\usepackage{algorithm}
\usepackage{algpseudocode}
\usepackage{tikz}
\usepackage{econometrics}
\usepackage{breqn}
\usepackage{enumerate}
\usepackage{subfig}
\usepackage[autostyle=false, style=english]{csquotes}
\usepackage{booktabs} 
\MakeOuterQuote{"}
\usepackage{amsthm}
\usepackage[toc,page]{appendix}

\theoremstyle{plain}
\newtheorem{theorem}{Theorem}

\newtheorem{lemma}{Lemma}

\theoremstyle{definition}

\newtheorem{remark}{Remark}

\title{Random Walks with Traversal Costs: Variance-Aware Performance Analysis and Network Optimization}

\author{ \href{https://orcid.org/0009-0008-1891-2346}{\includegraphics[scale=0.06]{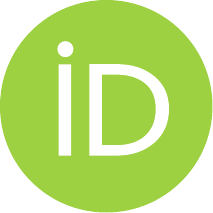}\hspace{1mm}Thao~Le}\\
	Department of Economics\\
	Vrije Universiteit Amsterdam\\
	\texttt{\href{mailto:t.p.t.le@vu.nl}{t.p.t.le@vu.nl}} \\
	\And
	\href{https://orcid.org/0009-0007-4945-6120}{\includegraphics[scale=0.06]{orcid.pdf}\hspace{1mm}Robbert~van~der~Burg} \\
	Department of Mathematics\\
	Vrije Universiteit Amsterdam\\
	\texttt{\href{mailto:r.vander.burg@vu.nl}{r.vander.burg@vu.nl}} \\
	\And
	\href{https://orcid.org/0000-0002-3389-2311}{\includegraphics[scale=0.06]{orcid.pdf}\hspace{1mm}Bernd~Heidergott} \\
	Department of Operations Analytics\\
	Vrije Universiteit Amsterdam\\
	\texttt{\href{mailto:b.f.heidergott@vu.nl}{b.f.heidergott@vu.nl}}
    \And
	\href{https://orcid.org/0000-0003-4072-3820}{\includegraphics[scale=0.06]{orcid.pdf}\hspace{1mm}Ines~Lindner} \\
	Department of Economics\\
	Vrije Universiteit Amsterdam\\
	\texttt{\href{mailto:i.d.lindner@vu.nl}{i.d.lindner@vu.nl}}
    \And
	\href{https://orcid.org/0000-0001-6585-4785}{\includegraphics[scale=0.06]{orcid.pdf}\hspace{1mm}Alessandro~Zocca} \\
	Department of Mathematics\\
	Vrije Universiteit Amsterdam\\
	\texttt{\href{mailto:a.zocca@vu.nl}{a.zocca@vu.nl}}
}

\renewcommand{\headeright}{ }
\renewcommand{\shorttitle}{Random Walks with Traversal Costs}

\hypersetup{
pdftitle={Random Walks with Traversal Costs},
pdfauthor={Thao Le, Robbert van der Burg, Bernd Heidergott, Ines Lindner, Alessandro Zocca}
}

\begin{document}
\maketitle

\begin{abstract}
	We introduce weighted Markovian graphs, a random walk model that decouples the transition dynamics of a Markov chain from (random) edge weights representing the cost of traversing each edge. This decoupling allows us to study the accumulated weight along a path independently of the routing behavior. Crucially, we derive closed-form expressions for the mean and variance of weighted first passage times and weighted Kemeny constants, together with their partial derivatives with respect to both the weight and transition matrices. These results hold for both deterministic and stochastic weights with no distributional assumptions. We demonstrate the framework through two applications, highlighting the dual role of variance. In surveillance networks, we introduce the surprise index, a coefficient-of-variation metric quantifying patrol unpredictability, and show how maximizing it yields policies that are both efficient and hard to anticipate. In traffic networks subject to cascading edge failures, we develop a minimal-intervention framework that adjusts speed limits to preserve connectivity under three increasingly flexible regulatory policies.
\end{abstract}

\keywords{Markov chains, weighted Markovian graphs, weighted random walks, weighted Kemeny constant, first passage times, constrained network optimization, network resilience}

\section{Introduction}
Graphs provide a natural framework for modeling complex systems, from social networks and urban traffic to communication infrastructure and the architectures underlying neural networks.
By representing entities as nodes and their interactions as edges, graphs enable a rigorous study of connectivity and dynamics.
A central analytical tool in this context is the \textit{random walk}, a stochastic process on the graph's nodes governed by a Markov chain with transition matrix $\mP$.
Random walks have a rich history dating back to \cite{pearson} and \cite{mccrea1936problem}, and have since found widespread application. The most prominent application is the PageRank algorithm \cite{brin1998anatomy,page1999pagerank}. Its variants have been applied in various settings ranging from web search \cite{fogaras2005towards,haveliwala2003topic,pan2004automatic,shen2014lazy} to complex social networks analysis \cite{jackson,sarkar2011random}; for a comprehensive overview, we refer to \cite{masuda2017random}.


Beyond the stationary distribution, the mean first passage times (MFPTs) of a random walk are a prominent tool in network analysis.
MFPTs quantify the expected number of steps for a walker to travel from one node to another, providing a natural notion of ``distance'' between nodes.
This interpretation has made MFPTs valuable for community detection (\cite{newman2004finding,zhang2013random}), social networks analysis (\cite{berkhout2019analysis,ballal2022network}), and cluster analysis in data science
(\cite{fortunato2010community,von2007tutorial}).
A quantity closely related to MFPTs is the Kemeny constant, which aggregates MFPTs into a single measure of overall network connectivity. Recent work has used the Kemeny constant for edge centrality analysis to identify bottleneck roads in traffic networks \cite{altafini2023edge} and to analyze its behavior under network reduction \cite{bini2024kemeny}.

Standard random walks track only \textit{where} a walker goes, not \textit{what it costs} to get there.
In many real-world systems, the probability of traversing a link and the cost (time, energy, monetary expense) associated with that traversal are distinct quantities.
Consider, for instance, a logistics network: a delivery driver may have a high probability of choosing a highway route, but that route may incur significant tolls or fuel costs.
Similarly, in communication networks, a routing protocol may prefer stable links, yet the quantities of interest are total latency or packet loss accumulated along a path.
In these settings, knowing the transition structure alone is insufficient; we must also account for the weights associated with each transition.


To address this gap, we study a non-standard random-walk model that assigns (random) weights to the edges, representing the cost of traversing each edge. 
While MFPTs count the mean number of steps from node $i$ to node $j$, the mean weighted first passage time (MWFPT) captures the mean weight accumulated over the first passage from node $i$ to $j$.
Closed-form solutions for MFPTs and the variance of first passage times (VFPTs) were established in the seminal work of \cite{kemeny1976finite}.
Alternative expressions for MFPTs and VFPTs for doubly stochastic chains can be found in \cite{borkar2012hamiltonian}, and for distributional results for lattice walks, we refer to \cite{lawler2010random}.
A closed-form expression for MWFPTs appears in \cite{patel2015robotic}. However, the algebraic form of that expression does not lend itself to efficient gradient computation, motivating the alternative formulation we develop in Section~\ref{sec:mvwpt}. Specifically, in our derivation, we use the fundamental matrix as a generalized inverse rather than the one used in \cite{patel2015robotic}, obtaining an expression for the MWFPT that is linear in the weights, a structural property enabling all subsequent gradient analysis.

In this paper, we also provide closed-form solutions for both MWFPT and the variance of the weighted first passage times (VWFPT), together with their partial derivatives with respect to $\mP$ and the first and second moments of the weights.
Building on the gradient-search optimization framework for Markov chains developed in \cite{franssen2025first}, we leverage these closed-form gradient expressions to optimize objectives that jointly account for mean performance and variability, a capability not available through standard MFPT-based approaches.


We demonstrate the versatility of this framework through two applications.
In \textit{traffic networks}, where nodes represent crossings and edges represent road segments, the transition matrix $\mP$ encodes the driver's routing behavior, while weights capture travel times.
We model the network's response to cascading edge failures, a setting also studied in the context of power grid resilience \citep{Bienstock_Verma_2010}, and show how optimally adjusting speed limits (i.e., edge weights) can preserve connectivity between critical nodes, where
connectivity is expressed via MWFPTs and VWFPTs. 
This application complements recent work on Kemeny-constant-based edge centrality for identifying critical infrastructure links \citep{altafini2023edge,bini2024kemeny}.

In \textit{surveillance networks}, the framework reveals a conceptual inversion of the classical mean-variance paradigm.
Although variance typically represents risk to be minimized \cite{Markowitz1952}, in adversarial patrol settings, unpredictability is an asset: a low-variance patrol schedule is exploitable by an intruder who can time their actions to coincide with predictable coverage gaps \cite{duan2021stochastic}.
Existing Markov chain-based surveillance strategies optimize either for efficiency, by minimizing the (weighted) Kemeny constant \cite{patel2015robotic,franssen2025first}; or for unpredictability, by maximizing the entropy rate \cite{george2018markov,duan2019markov}; see \cite{duan2021markov} for a survey.
By jointly optimizing the MWFPT and VWFPT, we derive patrolling strategies that maintain efficient coverage (i.e., low mean revisit time) while maximizing the variability of arrival times, providing a rigorous foundation for designing unpredictable security protocols.


The main contributions of this paper are as follows:
\begin{enumerate}
\item Using the fundamental matrix as a generalized inverse, we provide closed-form expressions for MWFPT and VWFPT.

\item We also derive partial derivatives of MWFPT and VWFPT with respect to the weight matrix $\mW$ and transition matrix $\mP$, requiring only finite first and second moments, and no distributional assumptions needed. These expressions accommodate both deterministic and stochastic weights.

\item We establish the smoothness of MWFPT and VWFPT with respect to the transition matrix $\mP$. Together with (ii), these results confirm that network objectives involving weighted first passage times are amenable to gradient-based optimization.

\item We demonstrate the versatility of our model through two numerical applications. First, in a surveillance network, we exploit VWFPTs to maximize patrol unpredictability, illustrating the novel role of variance as an asset rather than a risk measure. Second, in a traffic network subjected to cascading edge failures, we optimize travel-time connectivity under degradation.

\end{enumerate}


The remainder of the paper is organized as follows.
The non-standard weighted random walk model is presented in detail in Section~\ref{sec:wmp}, together with the mean and variance of the weighted first passage time.
Closed-form expressions for these quantities are derived in Section~\ref{sec:mvwpt}.
Partial derivatives with respect to mean weights and $\mP$ are established in Section~\ref{sec:gradw}.
Numerical examples that demonstrate the versatility of our approach are presented in Section~\ref{sec:app}. We conclude with an outlook for further research.

\section{The Weighted Markovian Graph Model}\label{sec:wmp}
This section establishes the weighted random walk framework utilized throughout our subsequent analysis. We explore three time-related physical interpretations of the weight matrix and formally define the mean weighted first passage time matrix. We ultimately derive two expressions for the weighted Kemeny constant to serve as objective measures of global connectivity.

\subsection{Definitions and Motivating Examples}\label{sec:ime}

The standard random walk is defined on a graph $ \hat {\cal G} = (\sV, \sE) $ with node set $ \sV = \{ 1 , \ldots , n \} $ and edge set $ \sE \subseteq \sV \times \sV $ with $|\sE|=m$, governed by a Markov transition matrix $ \mP$ on $ \sV$, dictating the probability of moving from one node to another in such a way that $\mP(i,j) \geq  0$ for $(i,j) \in \sE$ and $\mP(i,j) = 0$, otherwise, and $ \sum_j \mP ( i , j ) =1$ for all $ i \in \sV$. The Markov chain associated with $ \mP$ models a random walk on $ \hat {\cal G}$, and we refer to $ {\cal G} =  ( \sV, \sE , \mP )$ as a \textit{Markovian graph}.
For ease of presentation, we assume throughout the paper that $ {\cal G}$ is strongly connected and, therefore, $ \mP$ irreducible. Moreover, unless otherwise stated, we assume $ \mP$ to be aperiodic.

We augment the standard random walk by adding (random) weights $ {\cal W} \in \mathbb{R}^{n \times n} $ associated with the edges of the graph ${\cal G}$, such that $ {\cal W}  (i,j)  \in [ \epsilon , \infty )$ for all $ (i  , j ) \in \sE$, for $ \epsilon > 0 $, and zero otherwise.
These weights $ {\cal W} (i,j) $ are generally interpreted as the weight or cost associated with the random walker's transition from node $ i $ to node $j$.
We refer to the resulting structure $ {\cal G} =  ( \sV, \sE , \mP , {\cal W })$ as \textit{weighted Markovian graph}. Here, we explicitly decouple the transition dynamics governed by $\mP$ from the weights ${\cal W}$, which we allow to be an arbitrary matrix of costs.

In applications, it is particularly interesting to consider scenarios in which these weights model time in the system, where ``time'' is a placeholder for the more general concept of ``cost''.
 
We distinguish three primary application scenarios based on the temporal interpretation of $\mathcal{W}$:
\begin{enumerate}
\item {\bf Holding times:} In this scenario, the weight $\mathcal{W}(i, j) = \mathcal{H}(i, j)$ represents the holding time at node $i$ before the instantaneous transition to node $j$. 
The random walker moves from $i$ to $j$ with probability $\mathbf{P}(i, j)$ but remains at node $i$ for $\mathcal{H}(i, j)$ time units before an instantaneous jump.\\
\textit{{Example (marketing):} Let $\mathcal{V}$ be a set of products. $\mathbf{P}(i, j)$ is the probability that a customer who bought the product $i$ next buys the product $j$, and $\mathcal{H}(i, j)$ is the time the customer stays loyal to product $i$ before switching.}

\item {\bf Travel times:} Here, the weight $\mathcal{W}(i, j) = \mathcal{T}(i, j)$ represents the traversal time required for the transition itself.
The transition from node $i$ to $j$ is not instantaneous but takes $\mathcal{T}(i, j)$ time units.\\
\textit{{Example (road network):} In a road network, $\sV$ might represent a set of intersections where $\mathbf{P}(i, j)$ is the probability that a driver will move from intersection $i$ to $j$, and $\mathcal{T}(i, j)$ is the travel time between the intersections.}

\item {\bf Hybrid times:} In this scenario, the weight is described as $\mathcal{W}(i, j) = \mathcal{H}(i, j) + \mathcal{T}(i, j)$. This interpretation combines the two preceding concepts, allowing separate modeling of the wait time at a node and the travel time between nodes.
The total time associated with the transition is the sum of the holding time, $\mathcal{H}(i, j)$, and the traversal time, $\mathcal{T}(i, j)$.\\
\textit{{Example (road network):} $\mathcal{H}(i, j)$ represents the waiting time at intersection $i$ (e.g., at a red light) on the path to $j$, while $\mathcal{T}(i, j)$ is the actual travel time from $i$ to $j$.}
\end{enumerate}

\vspace{0.05cm}

A path $ \rho $ on $ {\cal G} = (\sV, \sE, \mP, {\cal W} )$ from $ i$ to $ j $ is a finite sequence $\rho = \{(i_1,j_1), \dots, (i_l,j_l)\}$ of links $ ( i_k , j_k ) \in \sE$, for $ 1 \leq k \leq l $, with $ j_k = i_{k+1}$, for $ 1 \leq k \leq l-1 $, $ i_1 = i$ and $ j_l = j $. We denote by $ l (  \rho )$ the length of path $ \rho$.
A path $ \rho $ from $ i $ to $ j $ is called a \textit{first passage path} if $ i_k \not = j$ for $ k=2, \ldots , l$, that is, if $ \rho $ is a finite path going from $ i $ and $ j $ that does not visit $ j$ before reaching the end node.
We denote the set of first passage paths from $ i $ to $ j $ by $ {\cal P} ( i , j ) $.
The probability of observing $ \rho \in  {\cal P} ( i , j ) $ is given by
\[
p ( \rho ) : = \prod_{k=1}^{ l ( \rho )} \mP ( i_k , j_k ) ,
\]
and the mean first-passage-path \textit{length} is now given by
\[
\mL ( i , j ) :=  \sum_{ \rho \in {\cal P} ( i , j ) }
l ( \rho ) \, p ( \rho ) .
\]
The accumulated weight along a path $ \rho $ is given by
\[
r ( \rho ) := \sum_{k=1}^{ l ( \rho )} {\cal W} ( i_k , j_k ) , 
\]
and the \textit{mean first-passage-path weight} is now given by
\begin{dmath*}
{\mM ( i , j ) := \mM_{\cal W} ( i , j ) }= \sum_{ \rho \in {\cal P} ( i , j ) }
r ( \rho ) \, p ( \rho ) 
 = \sum_{ \rho \in {\cal P} ( i , j ) }
\left ( \sum_{k=1}^{ l ( \rho )} {\cal W} ( i_k , j_k ) \right ) \prod_{k=1}^{ l ( \rho )} \mP ( i_k , j_k ) \,  .
\end{dmath*}
Setting all weights to one, $ \mM  $ recovers $\mL $; in formula, $ \mM_{ {\bf 1} {\bf 1}^\top} = \mL $, where ${\bf 1}$ denotes a column vector of ones, its dimension being specified by the context.
The \textit{variance of the first-passage-path weight} is given by 
\begin{dmath*}
{\mV ( i , j )  :=
\mV_{\cal W}(i,j)} = \left ( \sum_{ \rho \in {\cal P} ( i , j ) }
( r ( \rho ) )^2 \, p ( \rho )  \right ) -  \big (\mM_{\cal W}(i,j)  \big)^2  .
\end{dmath*}
Throughout the paper, we make the following assumption:

\vspace{0.2cm}

\begin{description}
    \item[{\bf (A)} ] \textit{The Markov chain $ \mP$ has a unique stationary distribution $ \vpi$. The weights have finite mean and second moment, denoted by $ \mW(i,j) = \mathbb{E} [ {\cal W} (i,j) ] < \infty$ and $ \mW^{(2)} ( i , j ) = \mathbb{E} [ {\cal W}^2 ( i , j ) ] < \infty $, $ ( i , j ) \in \sE$.}
       
\end{description}

\vspace{0.2cm}

It should be noted that assumption {\bf (A)} allows the weights ${\cal W}$ to be deterministic, in formula {$ \mW (i,j) = {\cal W} (i,j) $} for some $ ( i , j ) $. 

\subsection{Random Walk Analysis}
Let  $ \{ \hat X_k : k \geq 0 \} $ denote 
a discrete-time random walk on $ \hat {\cal G}$ with transition probabilities $  \mP(i,j) $.
Let 
\begin{align*}
    \hat \tau ( i , j ) = \inf &\{ n \geq 1 : \hat X_{k+n} = j, \hat X_{k+l} \neq  j \\
    & \text{ for } 1 \leq l \leq n-1, \hat X_k = i \}
\end{align*}
be the first passage time of $ \hat X_k$ from $ i $ to $ j$.
A random walker on $ \hat {\mathcal G}$ including these weights can now be defined by
\[
X_t = \hat X_{t_k} ,  \text{ for } t_k \leq t < t_{k+1}, \quad
\text{where } t_k = \sum_{i=0}^k \tau_i ,
\]
with $ \tau_0 = 0 $, and 
\[
\tau_k = {\cal W} ( \hat X_k , \hat X_{k+1} ) ,
\]
for $ k \geq 0 $.

Due to our assumption that $ {\cal W}( i , j ) > 0 $ for $ ( i , j ) \in \sE$,
we may interpret $  {\cal W}( i , j )$ as holding times, and the resulting ``continuous time'' process $ \{ X_t : t \geq 0 \}$ becomes a semi-Markov process.
By construction, the transition probabilities of $ X_t $ are given by
\[
\mP ( i , j ) = \mathbb{E} \left [  \left . \mathbb{I}_{ X_{ t_{k+1}} =j } \right |  X_{{t_k } = i } \right ] , \qquad 1 \leq i , j \leq n,
\]
and we refer to $ \hat X_k = X_{ t_k } $ as the \textit{embedded jump chain} of $ \{ X_t : t \geq 0 \} $ governed by $ \mP$.
In accordance with {\bf (A)}, $ \{ X_t : t \geq 0 \} $ has a unique stationary distribution $ \vpi$.
Furthermore, the accumulated weight along the first passage path $ \hat \tau ( i , j ) $ from $ i $ to $ j$ is given by
\begin{equation*}
\tau ( i , j ) =   \sum_{k=1}^{\hat \tau ( i , j ) }  \tau_k = \sum_{k=1}^{\hat \tau ( i , j ) }  {\cal W} ( \hat X_k , \hat X_{k+1} ),
\end{equation*}
provided that $  \hat X_k  = i$.
From the definitions, it readily follows that $ \mM ( i , j ) = \mathbb{E} [ \tau ( i , j ) ]$.

Let $\bar{U}(i)$ be the expected weight of state $i$ before transitioning to another state: 
\begin{align*}
    \bar{U}(i) 
       & = \sum_{k \in \sV} \mP(i,k) \, \E[ {\cal W} ( i , k  )  | X_{t-1}=i, X_t=k]\\
       & = \sum_{k \in \sV} \mP(i,k) \, \mW ( i,k) \\
       & = [(\mP \circ \mW) \mathbf{1}]_i.
\end{align*}
where $ \circ$ denotes the element-wise matrix product.
Then, the steady state probability, or long-run fraction of time spent in state $i$, is equal to:
\begin{equation}\label{pi_W}
    \vpi_{\cal W} (i) = \frac{\vpi(i) \bar{U}(i)}{\sum_{k \in \sV} \vpi(k) \bar{U}(k)}.
\end{equation}

It is easy to see that for given $ \mP$ (and thus given $ \vpi$),
any alternative stationary distribution $\vpi_{\cal W } $ with  $\vpi_{\cal W } >0$ can be achieved by choosing the appropriate mean weights $ \mW$. 
This illustrates additional flexibility offered by the weighted random walk. 

\subsection{Weighted Kemeny Constants}

The \textit{Kemeny constant} is a well-known connectivity measure of Markov chains, see \cite{kemeny1976finite} for an early reference, and is given by
\[
K ( \mP) = \sum_{ j \in \sV } \mL ( i , j  ) \vpi ( j ) ,
\]
for arbitrary $ i \in \sV$.
The quantity  $ K ( \mP )$ has been shown to be indeed independent of $ i $ in
\cite{kemeny1976finite}, but their result lacks an intuitive explanation.
In this paper, we will work with the
algebraically equivalent form 
\[
K ( \mP) = \sum_{ i , j \in \sV }  \vpi ( i ) \mL ( i , j  ) \vpi ( j ) .
\]
Recall that in the random walk model, $ \vpi ( i )$ expresses the (relative) importance of node $i$.
In words, $ \vpi ( i ) \mL ( i , j  ) \vpi ( j )$ thus weights the mean path length it takes the random walker to go from $ i $ to $j $ weighted by the importance of $ i $ and $ j $,
Hence, $ K ( \mP )$ is small if the paths between the important nodes are short.
The Kemeny constant is related to the \textit{effective graph resistance} or \textit{Kirchhoff index} given by 
\begin{equation*}
 R = \frac{1}{ 2| \sE |}  \sum_{i < j} \big (  \mL ( i , j ) + \mL ( j,i )  \big ).
\end{equation*}
The effective graph resistance is used as a proxy for robustness in the design and control of power transmission networks and is explored in-depth in \cite{ellens_2011}.

Replacing $ \mL$ by $ \mM$ and $ \vpi $ by $ \vpi_{\cal W}$, gives the following version of the Kemeny constant of weighted Markovian graphs  
\begin{equation}\label{eq:fok}
K_{\cal W} ( \mP, {\cal W}) = \sum_{i, j \in \sV }  \vpi_{\cal W} ( i ) \mM( i , j  ) \vpi_{\cal W} ( j ) .
\end{equation}
An equivalent, though less insightful, expression for $ K_{\cal W} ( \mP, {\cal W}) $ is ${\rm tr}(\mZ) \cdot \vpi (\mP \circ \mW) {\bf 1} $, for a proof we refer to Lemma~\ref{le:linK} later in the text.

Alternatively, we may define the Kemeny constant through the stationary distribution of the discrete-time random walk $ \vpi$:
\begin{dmath}\label{travel_Kemeny_def}
K  ( \mP, {\cal W}) = \sum_{i, j \in \sV }  \vpi ( i ) \mM( i , j  ) \vpi ( j ) .
\end{dmath}
The difference between $K_{ \cal W}  ( \mP, {\cal W}) $ and $ K  ( \mP, {\cal W}) $ lies in the choice of stationary distribution. The distinction between these two Kemeny constants has important implications for optimization, which we discuss in Section~\ref{sec:KK}.

The appropriate connectivity measure depends on the physical dynamics of the system: $K_{\cal W} (\mP, {\cal W})$ is the superior metric for holding-time models because it directly incorporates the duration $\bar U(i)$ spent in each state, whereas $K(\mP,{\cal W})$ better reflects connectivity in travel-time models where discrete event frequencies and routing volumes are the primary concern.
Moreover, as we show in Section~\ref{sec:KK}, $K_{ \cal W}  ( \mP, {\cal W}) $ and $ K  ( \mP, {\cal W}) $ behave fundamentally different in terms of their dependency on $ \mW$. For applications of $K_{ \cal W}  ( \mP, {\cal W}) $ and $ K  ( \mP, {\cal W}) $, we refer to Section~\ref{sec:app}. 

\section{Mean and Variance of Weighted First Passage Times}\label{sec:mvwpt}

We denote by $[\cdot]_{\rm{dg}}$ the operator that sets all off-diagonal elements of a given matrix to zero. For a given vector $ \vx$, the matrix $ {\rm{dg}} (\vx ) $ is defined as having diagonal elements $( {\rm{dg}} (\vx )) ( i,i ) = \vx ( i )  $ and zeros elsewhere.
\cite{kemeny1976finite} showed that
\begin{dmath}\label{eq:ks}
\mL = ( \mI - \mZ + {\bf 1} {\bf 1}^\top [ \mZ]_{ \rm{dg}} ) \Xi^{-1} ,
\end{dmath}
where we define $\Xi := {\rm dg}(\vpi)$ such that $\Xi^{-1} =  [ \mL ]_{ \rm{dg}}$, 
\begin{dmath*}\label{eq:propertyZ}
 \mZ = (\mI - \mP + \Pi )^{-1}
 =
 \sum_{ k=0}^\infty ( \mP - \Pi)^k = \Pi + \sum_{ k=0}^\infty ( \mP^k  - \Pi)  ,
\end{dmath*}
and $ \Pi = {\bf 1} \vpi $, also referred to as the \textit{ergodic projector}.
The matrix $ \mZ $ is called the \textit{group inverse} or \textit{fundamental matrix}, and its existence follows from finiteness and irreducibility of the embedded chain $ \mP$.
For an overview of various ways of computing $ \mL $, we refer to \cite{hunter2018computation}.
Specifically, we explain the relation of $ \mL$ and the so-called {\em taboo} kernel in Section~\ref{sec:taboo} in the Appendix.

The following theorem shows that $ \mM$ also admits a closed-form solution.

\begin{theorem}[MWFPT formula]\label{le:1}
If (A) holds, then the mean first-passage path weights satisfy
\[
 [ \mM ]_{\rm{dg}}  = \pi  ( \mP \circ \mW ) {\bf 1} [ \mM ]_{\rm{dg}} 
 = \pi  ( \mP \circ \mW ) {\bf 1} \Xi^{-1} ,
\]
and 
\begin{dmath*}
\mM = 
\Big ( 
\mZ  ( \mP \circ \mW ) \Pi 
-
{\bf 1} {\bf 1}^\top [ \mZ ( \mP \circ \mW ) \Pi ]_{\rm{dg}}
+
\pi ( \mP \circ \mW ) {\bf 1}
( \mI - \mZ + {\bf 1} {\bf 1}^\top [ \mZ]_{\rm{dg}}
\Big  )\Xi^{-1} .
\end{dmath*}
\end{theorem}

\noindent
{\bf Proof:}
The mean first-passage path weight satisfies the following recursive relation
\begin{dmath*}
\mM(i,j)  =  \mP ( i , j )  \mW ( i ,j ) +
\sum_{k \not = j }\mP ( i , k ) \big  ( \mW ( i ,k )  + \mM ( k , j )  \big )  =  - \mP ( i , j ) \mM ( j , j )  + \sum_{k  }\mP ( i , k )  \big ( \mW ( i ,k )  + \mM ( k , j ) \big ) 
  ,
\end{dmath*}
which in matrix notation reads as
\begin{equation*}
\mM = -\mP [ M ]_{\rm{dg}} + ( \mP \circ \mW ) {\bf 1} {\bf 1}^\top + \mP \mM,
\end{equation*}
or, equivalently,
\begin{equation}\label{eq:op}
( \mI - \mP ) \mM = -\mP [ \mM ]_{\rm{dg}} + ( \mP \circ \mW ) {\bf 1} {\bf 1}^\top .
\end{equation}
Multiplying both sides with $ \pi$ and using the fact that $ \pi ( \mI - \mP ) = 0$, we get 
\begin{equation}\label{eq:pi diag M}
\pi \mP [ \mM ]_{\rm{dg}}  =  \pi  [ \mM ]_{\rm{dg}}  =  \vpi ( \mP \circ \mW ) {\bf 1} {\bf 1}^\top .
\end{equation}
Multiplying by $ \Xi^{-1} $ on the left, using that $ \pi (i ) = 1 / \mL ( i , i ) $ and that $ \pi  ( \mP \circ \mW ) {\bf 1}$ is a scalar, we get 
\begin{equation}\label{eq:diagM}
[ \mM ]_{\rm{dg}} = \pi  ( \mP \circ \mW ) {\bf 1} \Xi^{-1} .
\end{equation}
Thus, all the terms on the right-hand side of \eqref{eq:op} are known.

We now use the following result from linear algebra (see \cite{hunter1,hunter2,hunter3}). A matrix $ \mA^-$ is called the generalized inverse of $ \mA $ if $ \mA \mA^- \mA = \mA $. 
Using the generalized inverse, it is easily seen that $ \mX = \mA^- \mC + ( \mI - \mA^- \mA ) \mU $, where $ \mU $ is an arbitrary matrix, solves $ \mA \mX = \mC$, provided that $\mA^-$ is consistent, which means that $ \mA^-$ satisfies
$ (\mI - \mA \mA^- ) \mC = 0$.
A generalized inverse of $ ( \mI - \mP) $ is given by 
\[
\mZ = ( \mI - \mP   + \Pi )^{-1} .
\]
In \eqref{eq:op}, we identify $ \mA = \mI - \mP, C =-  \mP [ \mM ]_{\rm{dg}}  + ( \mP \circ \mW ) {\bf 1} {\bf 1}^\top $ and $ \mA^- = \mZ $.
First, we show that $ \mZ $ is consistent, that is, it solves
\begin{equation}\label{eq:zmp}
(\mI  - \mZ (\mI - \mP )) \Big ( -  \mP [ \mM ]_{\rm{dg}} + ( \mP \circ \mW ) {\bf 1} {\bf 1}^\top  \Big ) 
= 0 .
\end{equation}
By simple algebra, the following expression holds
\[
 (\mI  -\mZ (\mI - \mP ) ) = (\mI  - (\mI - \mP )\mZ)  = \Pi .
\]
Inserting this into \eqref{eq:zmp} and combining with \eqref{eq:pi diag M} yields
\begin{dmath}\label{eq:d}
\Pi \mP [ \mM ]_{\rm{dg}}   = 
\Pi  [ \mM ]_{\rm{dg}}  =  
\Pi ( \mP \circ \mW ) {\bf 1} {\bf 1}^\top  .
\end{dmath}
By \eqref{eq:diagM} and noting that $ \pi ( \mP \circ \mW) {\bf 1}  $ is a scalar, it holds that for the left-hand side of \eqref{eq:d}
\begin{dmath*}
\Pi  [ \mM ]_{\rm{dg}}  =
\Pi  ( \pi ( \mP \circ \mW) {\bf 1} ) \Xi^{-1}  =
  ( \pi ( \mP \circ \mW) {\bf 1} )  \Pi \Xi^{-1}  =
\pi ( \mP \circ \mW) {\bf 1} ( {\bf 1} {\bf 1}^\top ) ,
\end{dmath*}
and, since $ \Pi$ has rows equal to $ \vpi $, for the right-hand side of \eqref{eq:d} it holds that
\[
 \Pi ( \mP \circ \mW ) {\bf 1} {\bf 1}^\top 
 = \pi ( \mP \circ \mW) {\bf 1} ( {\bf 1} {\bf 1}^\top ).
\]
This establishes the consistency of $ \mZ$, and shows that $\mZ $ is a generalized inverse of $ \mI -\mP $.

Next, we premultiply \eqref{eq:op} by $ \mZ $, and using the fact that $\mZ $ is a generalized inverse of $ \mI -\mP $ we get 
\begin{dmath*}
 \mM = \mZ ( ( \mP \circ \mW ) {\bf 1} {\bf 1}^\top    - \mP [ \mM ]_{\rm{dg}} )  + \Pi \mU.
\end{dmath*}
Since $ \Pi$ has equal rows, so has $ \Pi \mU $ and we can write
\begin{dmath}\label{eq:uu}
 \mM = 
\mZ \big ( ( \mP \circ \mW ) {\bf 1} {\bf 1}^\top    - \mP [ \mM ]_{\rm{dg}} \big )  + {\bf 1} \vu^\top ,
\end{dmath}
for some vector $ \vu $. It remains to solve $ \vu$. 
From \eqref{eq:uu} we obtain
\begin{dmath*}
{\rm{dg}} (\vu )  =  [ \mM ]_{\rm{dg}}  
- [ \mZ (( \mP \circ \mW ) {\bf 1} {\bf 1}^\top ]_{\rm{dg}}   + [ \mZ \mP ]_{\rm{dg}} [ \mM ]_{\rm{dg}} . 
\end{dmath*}
Recall that
\[
[ \mM ]_{\rm{dg}}  = \pi ( \mP \circ \mW ) {\bf 1} \Xi^{-1} ,
\]
which gives
\[
[ \mZ \mP ]_{\rm{dg}} [ \mM ]_{\rm{dg}} =
\pi ( \mP \circ \mW ) {\bf 1} [ \mZ \mP ]_{\rm{dg}} \Xi^{-1}   .
\]
With this, deduce
\begin{dmath*}
{\rm{dg}} (\vu ) = \pi ( \mP \circ \mW ) {\bf 1} \Xi^{-1} 
- [ \mZ (  \mP \circ \mW ) \Pi \Xi^{-1} ]_{\rm{dg}}
+  \pi ( \mP \circ \mW ) {\bf 1} [ \mZ \mP ]_{\rm{dg}} \Xi^{-1} =
\left ( \pi ( \mP \circ \mW ) {\bf 1}  
- [ \mZ ( ( \mP \circ \mW ) \Pi ]_{\rm{dg}}
+  \pi ( \mP \circ \mW ) {\bf 1} [ \mZ \mP ]_{\rm{dg}} \right  )\Xi^{-1} .
\end{dmath*}
This leads to the following representation of $ \mU$:
\begin{dmath*}
\Pi \mU = 
{\bf 1} {\bf 1}^\top  \left ( \pi ( \mP \circ \mW ) {\bf 1}  
- [ \mZ ( ( \mP \circ \mW ) \Pi ]_{\rm{dg}}
+  \pi ( \mP \circ \mW ) {\bf 1} [ \mZ \mP ]_{\rm{dg}} \right  )\Xi^{-1} .
\end{dmath*}
Moreover,
\[
{\bf 1} {\bf 1}^\top  = \Pi \Xi^{-1}  .
\]
Inserting the above equality into \eqref{eq:uu} gives
\begin{dmath}
\mM =  \nonumber
\mZ \big ( ( \mP \circ \mW ) {\bf 1} {\bf 1}^\top- \mP [ \mM ]_{\rm{dg}} \big )  +  {\bf 1} {\bf 1}^\top
\left ( \pi ( \mP \circ \mW ) {\bf 1}  
- [ \mZ (  \mP \circ \mW ) \Pi ]_{\rm{dg}}
+    \pi ( \mP \circ \mW ) {\bf 1} [ \mZ \mP ]_{\rm{dg}} \right  )\Xi^{-1} \nonumber
=
\mZ \big ( ( \mP \circ \mW ) \Pi \Xi^{-1}  - \pi ( \mP \circ \mW ) {\bf 1}  \mP \Xi^{-1}
 +  {\bf 1} {\bf 1}^\top
\left ( \pi ( \mP \circ \mW ) {\bf 1}  
- [ \mZ ( ( \mP \circ \mW ) \Pi ]_{\rm{dg}}
+    \pi ( \mP \circ \mW ) {\bf 1} [ \mZ \mP ]_{\rm{dg}} \right  )\Xi^{-1}   \nonumber
= \Big (\mZ  ( \mP \circ \mW ) \Pi   - \pi ( \mP \circ \mW ) {\bf 1}  \mZ \mP 
 +  {\bf 1} {\bf 1}^\top
\left ( \pi ( \mP \circ \mW ) {\bf 1}  
- [ \mZ (  \mP \circ \mW ) \Pi ]_{\rm{dg}}
+    \pi ( \mP \circ \mW ) {\bf 1} [ \mZ \mP ]_{\rm{dg}} \right  )  \Big ) \Xi^{-1} \nonumber
= \Big (\mZ  ( \mP \circ \mW ) \Pi  
-  {\bf 1} {\bf 1}^\top  [ \mZ  ( \mP \circ \mW ) \Pi ]_{\rm{dg}}
+\pi ( \mP \circ \mW ) {\bf 1}
\big ( {\bf 1} {\bf 1}^\top 
-  \mZ \mP +
 [ \mZ \mP ]_{\rm{dg}} \big  )  \Big ) \Xi^{-1}. \label{eq:all1}
\end{dmath}
By some further algebra, it can be shown that
\[
{\bf 1} {\bf 1}^\top - \mZ \mP  + [ \mZ \mP ]_{\rm{dg}}
= \mI - \mZ + {\bf 1} {\bf 1}^\top [ \mZ ]_{\rm{dg}}  .
\]
Inserting this into \eqref{eq:all1}, we arrive at the final representation:
\begin{dmath}\label{finalized_M_formula}
\mM = 
\Big ( 
\mZ  ( \mP \circ \mW ) \Pi 
-
{\bf 1} {\bf 1}^\top [ \mZ ( \mP \circ \mW ) \Pi ]_{\rm{dg}}
+
\pi ( \mP \circ \mW ) {\bf 1}
( \mI - \mZ + {\bf 1} {\bf 1}^\top [ \mZ]_{\rm{dg}}
\Big  )\Xi^{-1} .
\end{dmath}
\qed

\begin{remark}\label{re:2}\textit{
Elaborating on $ (\mI - \mP +(\mP \circ \mW ) {\bf 1} \vpi)^{-1} $ as generalized inverse of $ \mI - \mP $, an alternative expression for $ \mM$ is obtained in \cite{patel2015robotic}. However, choosing $ \mZ $ for the generalized inverse allows us to derive an expression for $ \mM $ that is linear in $ \mW $. This simplicity of $ \mM$ in terms of $ \mW $ enables the further analysis presented in this paper.}
\end{remark}

Let $\mM^{(2)} ( i , j ) :=\mM_{{\cal W}}^{(2)} ( i , j )  $ denote the second moment of the first passage time $ \tau ( i , j ) $.
Then, 
\begin{dmath*} \label{eq: variance}
    \mV_{{\cal W}}(i,j)   = \mM^{(2)} ( i , j ) - ( \mM ( i , j ) )^{2}
\end{dmath*}
is an expression for the variance of the first passage times.

The following theorem provides a closed-form expression for this variance. We need the following additional assumption of the weight matrix for our variance result.
\begin{description}
    \item[{\bf (A$^*$)} ] \textit{The sequence of weight matrices $ \{({\cal W} (i,j ; n ))_{ i , j } : n \geq 1 \}  $ are i.i.d.}
\end{description}
Assumption (A$^*$) entails that in the case of stochastic weights, the cost of traversing an edge is drawn from the same underlying distribution each time it is visited, independent of prior traversals. This naturally accommodates applications where the memoryless property holds, such as telecommunication networks, where packet delays vary due to independent background noise, or logistics networks, where shipping times fluctuate based on daily conditions.

\begin{theorem}[Second moment of weighted first passage times]\label{th: second momemt}
Provided that (A) and (A$^*$) hold, the second moment of the first passage time is equal to:
\begin{dmath*}
\mM^{(2)} ( i , i )  
= \frac{ 
\vpi ( \mP \circ \mW^{(2)} )  {\bf 1}
+
\big ( 2 \vpi  (\mP   \circ  \mW ) \big (   \mM - [ \mM]_{ \rm{dg}}  \big )  
\big )  ( i )  }{ \vpi ( i ) },
\end{dmath*}
where  
\begin{dmath*}
\mM^{(2)} = 
\mZ ( \mP \circ \mW^{(2)} ) {\bf 1} {\bf 1}^\top
-
{\bf 1} {\bf 1}^\top \Big[ \mZ ( \mP \circ \mW^{(2)}) {\bf 1} {\bf 1}^\top \Big] 
 +
2 \Big( \mZ ( \mP \circ \mW) \big( \mM - [ \mM]_{ \rm{dg}}\big)
-
{\bf 1} {\bf 1}^\top \Big[ \mZ ( \mP \circ \mW) \big( \mM - [ \mM]_{ \rm{dg}}\big) \Big]_{\rm{dg}} \Big)
+
\big( \mI - \mZ + {\bf 1} {\bf 1}^\top [\mZ]_{\rm{dg}} \big) \big[ \mM^{(2)}\big]_{\rm{dg}}\, . 
\end{dmath*}
\end{theorem}
The proof of this result elaborates on the generalized inverse and is analogous to the proof of Theorem~\ref{le:1}. For details, we refer to Section~\ref{sec:prooftheorem2} in the Appendix.

The mean weighted first passage times inspire the definition of the two versions $K_{ \cal W}  ( \mP, {\cal W}) $ and $ K  ( \mP, {\cal W}) $ for the Kemeny constant for a weighted Markovian Graph ${\cal G} = ( \mathbb{V} , \mathbb{E} , \mP , {\cal W })$.
Following the same line of thought, we now define
the \textit{second-order Kemeny constants} of 
the Markovian graph ${\cal G} $ as
\begin{equation}\label{eq:sok}
V_{\cal W} ( \mP, {\cal W}) := \sum_{i, j \in \sV }  \vpi_{\cal W} ( i ) \mV( i , j  ) \vpi_{\cal W} ( j ),
\end{equation}
and
\begin{equation}\label{sec-order Kemeny}
V  ( \mP, {\cal W}):= \sum_{i, j \in \sV }  \vpi ( i ) \mV( i , j  ) \vpi ( j ) .
\end{equation}
We refer to $K_{ \cal W}  ( \mP, {\cal W}) $ and $ K  ( \mP, {\cal W}) $ as
 \textit{first-order Kemeny constants}.
\color{black}

\section{Gradients with respect to the Mean Weight Matrix}
\label{sec:gradw}
In this section, we will derive closed-form expressions for the derivatives of $\vpi_{\cal W}$ and $\mM$. Moreover, we reflect on the relation of the choice of the weighted Kemeny constant to the nature of the application.

\subsection{Gradient Expressions}

We start by computing the gradient of $\vpi_{\cal W}$ with respect to the mean weights $\mW$.

\begin{lemma}[Derivative of $\vpi_{\cal W}$ with respect to $\mW$] \label{lem:derivative_pi_w}
Provided that (A) holds, the partial derivatives of the stationary weight distribution with respect to the mean weights are given by
\[
\frac{\partial \vpi_{ \cal W}(i)}{\partial \mW(l,k)} = 
\begin{cases}
\displaystyle\frac{\vpi(i) \mP(i,k)}{Y^2}\left(Y - \vpi(i)\bar{U}(i)\right) & \text{if } l = i , \\[10pt]
\displaystyle-\frac{\vpi(i)\bar{U}(i)\vpi(l)\mP(l,k)}{Y^2} & \text{if } l \neq i ,
\end{cases}
\]
where $Y = \sum_{j \in \sV} \vpi(j)\bar{U}(j)$ is a normalization constant.
\end{lemma}
The proof of the lemma is provided in Appendix~\ref{proof of lemma 4}.

Next, we compute the gradient of the MWFPTs with respect to the mean weights. The gradient is established in the following lemma,
where $\mathbf{e}_l$ denotes the $l$-th standard basis vector (column vector with 1 in position $l$, zeros elsewhere). 

\begin{lemma}[Derivative of First Moment, Mean Weights]\label{lem:derivative_mfpt}
Provided that (A) holds, the partial derivative of MWFPT with respect to $\mW(l,k)$ is:
\begin{dmath*}\label{eq:derivative_mfpt}
\frac{\partial \mM}{\partial \mW(l,k)} = \mP(l,k)\left[\mZ\mathbf{e}_l\mathbf{1}^\top - \mathbf{1}\mathbf{1}^\top [\mZ\mathbf{e}_l\mathbf{1}^\top]_{\rm{dg}} + \vpi(l)(\mI - \mZ + \mathbf{1}\mathbf{1}^\top [\mZ]_{\rm{dg}})\Xi^{-1}\right] .
\end{dmath*}
\end{lemma}

\noindent
{\bf Proof:}
The closed-form solution for $\mM$ using the fundamental matrix is:
\begin{dmath*}
    \mM = [\mZ(\mP \circ \mW)\mathbf{1}\vpi - \mathbf{1}\mathbf{1}^\top [\mZ(\mP \circ \mW)\mathbf{1}\vpi]_{\rm{dg}} + C(\mI - \mZ + \mathbf{1}\mathbf{1}^\top [\mZ]_{\rm{dg}})]\Xi^{-1} ,
\end{dmath*}
where $C = \vpi(\mP \circ \mW)\mathbf{1}$.
Taking the partial derivative with respect to $\mW(l,k)$:
\begin{dmath*}
    \frac{\partial \mM}{\partial \mW(l,k)} = \left[\mZ\frac{\partial(\mP \circ \mW)}{\partial \mW(l,k)}\mathbf{1}\vpi - \mathbf{1}\mathbf{1}^\top \left[\mZ\frac{\partial(\mP \circ \mW)}{\partial \mW(l,k)}\mathbf{1}\vpi\right]_{\rm{dg}} + \frac{\partial C}{\partial \mW(l,k)}(\mI - \mZ + \mathbf{1}\mathbf{1}^\top [\mZ]_{\rm{dg}})\right]\Xi^{-1} .
\end{dmath*}
Since $\frac{\partial(\mP \circ \mW)}{\partial \mW(l,k)}$ is a matrix with $\mP(l,k)$ in position $(l,k)$ and zeros elsewhere, we have $\frac{\partial(\mP \circ \mW)}{\partial \mW(l,k)}\mathbf{1} = \mP(l,k)\mathbf{e}_l$. Also, $\frac{\partial C}{\partial \mW(l,k)} = \vpi(l)\mP(l,k)$.
Substituting:
\begin{dmath*}
    \frac{\partial \mM}{\partial \mW(l,k)} = [\mZ\mP(l,k)\mathbf{e}_l\vpi - \mathbf{1}\mathbf{1}^\top [\mZ\mP(l,k)\mathbf{e}_l\vpi]_{\rm{dg}} + \vpi(l)\mP(l,k)(\mI - \mZ + \mathbf{1}\mathbf{1}^\top [\mZ]_{\rm{dg}})]\Xi^{-1} .
\end{dmath*}
Factoring out $\mP(l,k)$:
\begin{dmath*}
    \frac{\partial \mM}{\partial \mW(l,k)} = \mP(l,k)[\mZ\mathbf{e}_l\vpi - \mathbf{1}\mathbf{1}^\top [\mZ\mathbf{e}_l\vpi]_{\rm{dg}} + \vpi(l)(\mI - \mZ + \mathbf{1}\mathbf{1}^\top [\mZ]_{\rm{dg}})]\Xi^{-1}
\end{dmath*}
Now, note that $\mathbf{e}_l\vpi\Xi^{-1} = \mathbf{e}_l\mathbf{1}^\top$ since $\vpi\Xi^{-1} = \mathbf{1}^\top$. Therefore:
$ 
    \mZ\mathbf{e}_l\vpi\Xi^{-1} = \mZ\mathbf{e}_l\mathbf{1}^\top,
$ 
which gives as final result
\begin{dmath*}
    \frac{\partial \mM}{\partial \mW(l,k)} = \mP(l,k)\left[\mZ\mathbf{e}_l\mathbf{1}^\top - \mathbf{1}\mathbf{1}^\top [\mZ\mathbf{e}_l\mathbf{1}^\top]_{\rm{dg}} + \vpi(l)(\mI - \mZ + \mathbf{1}\mathbf{1}^\top [\mZ]_{\rm{dg}})\Xi^{-1}\right].
\end{dmath*}
\hfill $\square$

We now turn to the derivatives of the VWFPTs with respect to the mean weights.

\begin{theorem}[Derivative of VWFPTs with respect to $\mW$]\label{thm:variance_derivative}
Provided that (A) and (A$^*$) hold, the derivative of the VWFPTs with respect to $\mW(l,k)$ is:
\begin{equation*}\label{eq:variance_derivative}
\frac{\partial \mV_{\cal W} }{\partial \mW(l,k)} = \frac{\partial \mM^{(2)}}{\partial \mW(l,k)} - 2\mM \circ \frac{\partial \mM}{\partial \mW(l,k)} ,
\end{equation*}
where:
\begin{dmath}
\frac{\partial \mM}{\partial \mW(l,k)} = \mP(l,k)\left[\mZ\mathbf{e}_l\mathbf{1}^\top - \mathbf{1}\mathbf{1}^\top [\mZ\mathbf{e}_l\mathbf{1}^\top]_{\rm{dg}} + \vpi(l)(\mI - \mZ + \mathbf{1}\mathbf{1}^\top [\mZ]_{\rm{dg}})\Xi^{-1}\right] \label{eq:dMT}
\end{dmath}
and the derivatives of the second moment are:
\begin{dmath}
\frac{\partial \mM^{(2)}}{\partial \mW(l,k)} = \left[\mZ\frac{\partial(\mP \circ \mW^{(2)})}{\partial \mW(l,k)}\mathbf{1}\mathbf{1}^\top - \mathbf{1}\mathbf{1}^\top \left[\mZ\frac{\partial(\mP \circ \mW^{(2)})}{\partial \mW(l,k)}\mathbf{1}\mathbf{1}^\top\right]_{\rm{dg}}\right]
+ 2\mP(l,k)\left[\mZ\mathbf{e}_l\mathbf{e}_k^\top(\mM - [\mM]_{\rm{dg}}) - \mathbf{1}\mathbf{1}^\top [\mZ\mathbf{e}_l\mathbf{e}_k^\top(\mM - [\mM]_{\rm{dg}})]_{\rm{dg}}\right]  
+ 2\left[\mZ(\mP \circ \mW)\left(\frac{\partial \mM}{\partial \mW(l,k)} - \left[\frac{\partial \mM}{\partial \mW(l,k)}\right]_{\rm{dg}}\right)
- \mathbf{1}\mathbf{1}^\top \left[\mZ(\mP \circ \mW)\left(\frac{\partial \mM}{\partial \mW(l,k)} - \left[\frac{\partial \mM}{\partial \mW(l,k)}\right]_{\rm{dg}}\right)\right]_{\rm{dg}}\right]
+ (\mI - \mZ + \mathbf{1}\mathbf{1}^\top [\mZ]_{\rm{dg}})\frac{\partial [\mM^{(2)}]_{\rm{dg}}}{\partial \mW(l,k)} . \label{eq:dM2T}
\end{dmath}
The diagonal derivative appearing in~\eqref{eq:dM2T} can be computed explicitly as:
\begin{dmath}
\frac{\partial \mM^{(2)}(i,i)}{\partial \mW(l,k)} = \frac{\vpi\frac{\partial(\mP \circ \mW^{(2)})}{\partial \mW(l,k)}\mathbf{1}}{\vpi(i)}   + \frac{2\vpi(l)\mP(l,k)[\mM - [\mM]_{\rm{dg}}](i,k)}{\vpi(i)} + \frac{2\left[\vpi(\mP \circ \mW)\left(\frac{\partial \mM}{\partial \mW(l,k)} - \left[\frac{\partial \mM}{\partial \mW(l,k)}\right]_{\rm{dg}}\right)\right]_i}{\vpi(i)}  .\label{eq:dM2T_diag}
\end{dmath}
\end{theorem}

\noindent {\bf Proof:}
The VWFPTs is defined as
$ \mV_{\cal W} = \mM^{(2)} - (\mM \circ \mM) $.
From the Hadamard product rule $\frac{\partial(\mM \circ \mM)}{\partial \mW(l,k)} = 2\mM \circ \frac{\partial \mM}{\partial \mW(l,k)}$, we obtain:
$$\frac{\partial \mV_{\cal W}}{\partial \mW(l,k)} = \frac{\partial \mM^{(2)}}{\partial \mW(l,k)} - 2\mM \circ \frac{\partial \mM}{\partial \mW(l,k)} .$$

The expression in \eqref{eq:dMT} follows from differentiating the closed-form solution for $\mM$ in Theorem~\ref{le:1}. Taking the partial derivative with respect to $\mW(l,k)$:
\begin{dmath*}
\frac{\partial \mM}{\partial \mW(l,k)} = \left[\mZ\frac{\partial(\mP \circ \mW)}{\partial \mW(l,k)}\mathbf{1}\vpi - \mathbf{1}\mathbf{1}^\top\left[\mZ\frac{\partial(\mP \circ \mW)}{\partial \mW(l,k)}\mathbf{1}\vpi\right]_{\rm{dg}} + \frac{\partial C}{\partial \mW(l,k)}(\mI - \mZ + \mathbf{1}\mathbf{1}^\top[\mZ]_{\rm{dg}})\right]\Xi^{-1} .
\end{dmath*}
Since $\frac{\partial(\mP \circ \mW)}{\partial \mW(l,k)}$ has $\mP(l,k)$ in position $(l,k)$ and zeros elsewhere, we have $\frac{\partial(\mP \circ \mW)}{\partial \mW(l,k)}\mathbf{1} = \mP(l,k)\mathbf{e}_l$ and $\frac{\partial C}{\partial \mW(l,k)} = \vpi(l)\mP(l,k)$. Substituting and factoring out $\mP(l,k)$:
\begin{dmath*}
\frac{\partial \mM}{\partial \mW(l,k)} = \mP(l,k)\left[\mZ\mathbf{e}_l\vpi - \mathbf{1}\mathbf{1}^\top[\mZ\mathbf{e}_l\vpi]_{\rm{dg}} + \vpi(l)(\mI - \mZ + \mathbf{1}\mathbf{1}^\top[\mZ]_{\rm{dg}})\right]\Xi^{-1} .
\end{dmath*}
Applying $\vpi\Xi^{-1} = \mathbf{1}^\top$ to the first two terms gives $\mZ\mathbf{e}_l\mathbf{1}^\top - \mathbf{1}\mathbf{1}^\top[\mZ\mathbf{e}_l\mathbf{1}^\top]_{\rm{dg}}$, yielding equation~\eqref{eq:dMT}.

For the second moment derivative, we differentiate the closed-form expression for $\mM^{(2)}$ from Theorem~\ref{th: second momemt} term by term. 
The resulting formula contains $\left[\frac{\partial \mM^{(2)}}{\partial \mW(l,k)}\right]_{\rm{dg}}$ implicitly on the right-hand side. 
We resolve this by differentiating the explicit diagonal formula for $\mM^{(2)}(i,i)$ from Theorem~\ref{th: second momemt} directly, which yields equation~\eqref{eq:dM2T_diag}, an explicit expression depending only on known quantities and $\frac{\partial \mM}{\partial \mW(l,k)}$. Substituting these diagonal elements back completes the derivative of $\mM^{(2)}$; for details, see Section~\ref{derivative_second_moment} of the Appendix. \qed

\begin{remark}\label{rem:deterministic_rewards}
The derivative $\frac{\partial(\mP \circ \mW^{(2)})}{\partial \mW(l,k)}$ depends on the relationship between the second moment $\mW^{(2)}(i,j) = \mathbb{E}[\mathcal{W}^2(i,j)]$ and the mean $\mW(i,j) = \mathbb{E}[\mathcal{W}(i,j)]$, which is distribution-dependent. For deterministic weights, where $\mW^{(2)}(i,j) = (\mW(i,j))^2$, this simplifies to:
\begin{equation}
\frac{\partial(\mP \circ \mW^{(2)})}{\partial \mW(l,k)} = 2\mW(l,k)\mP(l,k)\mathbf{e}_l\mathbf{e}_k^\top
\end{equation}
and consequently $\frac{\partial(\mP \circ \mW^{(2)})}{\partial \mW(l,k)}\mathbf{1} = 2\mW(l,k)\mP(l,k)\mathbf{e}_l$. For stochastic weights,
the relationship between $\mW^{(2)}$ and $\mW$ differs, leading to alternative expressions for this derivative.
\end{remark}

Following similar lines of reasoning, the gradients with respect to $\mP$ can be derived -- though this entails dealing with the derivative of the stochastic matrix $\mP$, which is nontrivial -- and we refer to Section~\ref{sec:gradp} in the Appendix for details. 

\subsection{A Reflection on the Choice of Kemeny Constant}\label{sec:KK}

We first show that $ K ( \mP , {\cal W} )$ is linear in $ \mW $.

\begin{lemma}[Linearity of $K(\mP, {\cal W})$ in $\mW$]\label{le:linK}
For the Kemeny constant, it holds that
\begin{dmath*}
     K ( \mP , {\cal W} ) = {\rm tr}(\mZ) \cdot \vpi (\mP \circ \mW) {\bf 1}  \, .
\end{dmath*}
\end{lemma}

\noindent
{\bf Proof}: By substituting \eqref{finalized_M_formula} into \eqref{travel_Kemeny_def}, the following closed-form representation of the Kemeny constant can be derived:
\begin{dmath}\label{kemeny constant formula explicit}
    K(\mP, {\cal W}) = \vpi(\mP \circ \mW)\mathbf{1}(1 + {\rm tr}(\mZ)) - {\rm tr}(\Pi(\mP \circ \mW)) .
\end{dmath}
It remains to show that ${\rm tr}(\Pi(\mP \circ \mW)) = \vpi(\mP \circ \mW)\mathbf{1}$. Since $\Pi = \mathbf{1}\vpi$, every row of $\Pi$ equals $\vpi$, so:
\begin{dmath}\label{temp}
    {\rm tr}(\Pi(\mP \circ \mW)) = \sum_i \sum_j \vpi(j)\mP(j,i)\mW(j,i)
    = \sum_j \vpi(j) \sum_i \mP(j,i)\mW(j,i)
    = \vpi(\mP \circ \mW)\mathbf{1}.
\end{dmath}
Substituting \eqref{temp} back into \eqref{kemeny constant formula explicit} results in $K(\mP, {\cal W}) = {\rm tr}(\mZ) \cdot \vpi(\mP \circ \mW)\mathbf{1}$.
\hfill $\square$

Lemma~\ref{le:linK} allows for the following differentiation result.

\begin{theorem}[Derivative of $K(\mP, {\cal W})$ with respect to $\mW$] \label{theorem: derivative K^t with pi}
\noindent 
Provided that (A) holds, the derivative of the Kemeny constant $ K ( \mP , {\cal W} ) $ with respect to the weights $\mW$ is given by
$$
\frac{\partial K ( \mP , {\cal W} ) }{\partial \mW ( l,k ) } = 
\text{tr}(\mZ) \vpi (l) \mP ( l,k ),
$$
for $ (l,k ) \in \mathbb{E}$.
\end{theorem}

\noindent {\bf Proof}:
Starting from
$K(\mP, {\cal W}) = \beta(1 + \text{tr}(\mZ)) - \text{tr}(\vpi(\mP \circ \mW))$, 
where $\beta = \vpi(\mP \circ \mW)1$, we may take the derivative with respect to $\mW(l,k)$:

\begin{dmath*}
    \left(\frac{\partial K(\mP, {\cal W})}{\partial \mW}\right)_{(l,k)} = \frac{\partial \beta}{\partial \mW(l,k)}(1 + \text{tr}(\mZ)) + \beta \frac{\partial \text{tr}(\mZ)}{\partial \mW(l,k)} - \frac{\partial}{\partial \mW(l,k)}\text{tr}(\vpi(\mP \circ \mW))
\end{dmath*}
Since $Z = (I -\mP+ \vpi)^{-1}$ does not depend on $\mW$, its derivative w.r.t. $\mW$ is zero.
Then, for the derivative the following equation holds:
\begin{dmath*}
    \frac{\partial K(\mP, {\cal W})}{\partial \mW(l,k)} = \vpi(l) \mP(l,k)(1 + \text{tr}(\mZ)) - \vpi(l) \mP(l,k) = \vpi(l) \mP(l,k) \cdot \text{tr}(\mZ),
\end{dmath*}
or, in matrix notation,
$\partial K(\mP,{\cal W})/ \partial \mW  = \text{tr}(\mZ) \cdot (\vpi^\top \circ \mP).$
\hfill $\square$

The above theorem shows that the partial derivatives of
$ K ( \mP , {\cal W} ) $ do not depend on $\mW$. Indeed, $\mZ = (\mI -\mP+ \Pi)^{-1}$, $\Pi = {\bf 1} \pi$, and $\mP$ are all independent of the weight matrix $\mW$, the gradient depends only on the topology and stationary distribution of the Markov chain, not on the actual weights.
This independence of the actual weights has practical implications. Consider two networks with identical topology but different weights.
If we take $ K ( \mP , {\cal W} ) $  as a measure of connectivity for the weighted Markovian Graphs, then, by Theorem~\ref{theorem: derivative K^t with pi}, $ K ( \mP , {\cal W}) $ has identical gradients for both networks.
This makes it impossible to distinguish between edges with high weights and those with small weights in optimization w.r.t.~edge weights.
Next, we show that the gradient of $ K_{\cal W} ( \mP , {\cal W} ) $ overcomes this limitation as it depends on the weights.

\begin{theorem}[Derivative of $K_{\cal W} (\mP, {\cal W})$ with respect to $\mW$]\label{thm:derivative_kw}
Provided that (A) and (A$^*$) hold,
the derivative of the Kemeny $ K_{\cal W} ( \mP , {\cal W} )  $ constant with respect to $\mW(l,k)$ is:
\begin{dmath}
\frac{\partial K_{\cal W} ( \mP , {\cal W} ) }{\partial \mW(l,k)} = \frac{\vpi(l)\mP(l,k)}{Y^2}\left[(Y\mathbf{e}_l - \mathbf{n})^\top \mM \vpi_{ \cal W}^\top + \vpi_{ \cal W} \mM (Y\mathbf{e}_l - \mathbf{n})\right]  + \vpi_{ \cal W} \frac{\partial \mM}{\partial \mW(l,k)} \vpi_{ \cal W}^\top ,
\end{dmath}
where:
\begin{itemize}

\item $\mathbf{n} = (\vpi(1)\bar{U}(1), \ldots, \vpi(n)\bar{U}(n))^\top$,
\item $Y = \sum_{i \in \sV} \vpi(i)\bar{U}(i)$,
\item $\frac{\partial \mM}{\partial \mW(l,k)}$ is given by Lemma \ref{lem:derivative_mfpt}.
\end{itemize}
\end{theorem}

\noindent
{\bf Proof}:
 Apply the product rule to $K_{\cal W} ( \mP , {\cal W} ) = \vpi_{ \cal W} \mM \vpi_{ \cal W}^\top$:
\begin{dmath*}
\frac{\partial K_{\cal W} ( \mP , {\cal W} ) }{\partial \mW(l,k)} = \frac{\partial \vpi_{ \cal W}}{\partial \mW(l,k)} \mM \vpi_{ \cal W}^\top + \vpi_{ \cal W} \frac{\partial \mM}{\partial \mW(l,k)} \vpi_{ \cal W}^\top + \vpi_{ \cal W} \mM \frac{\partial \vpi_{ \cal W}^\top}{\partial \mW(l,k)}.
\end{dmath*}
From Lemma \ref{lem:derivative_pi_w}, we can express the derivative in vector form as:
\begin{dmath*}
    \frac{\partial \vpi_{ \cal W}}{\partial \mW(l,k)} = \frac{\vpi(l)\mP(l,k)}{Y^2}\left[Y\mathbf{e}_l - \mathbf{n}\right] ,
\end{dmath*}
where 
$\mathbf{n} = (\vpi(1)\bar{U}(1), \ldots, \vpi(n)\bar{U}(n))^\top$.

For $i = l$, the $l$-th component is $\frac{\vpi(l)\mP(l,k)}{Y^2}[Y \cdot 1 - \vpi(l)\bar{U}(l)] = \frac{\vpi(l)\mP(l,k)}{Y^2}[Y - \vpi(l)\bar{U}(l)]$
For $i \neq l$, the $i$-th component is $\frac{\vpi(l)\mP(l,k)}{Y^2}[Y \cdot 0 - \vpi(i)\bar{U}(i)] = -\frac{\vpi(i)\bar{U}(i)\vpi(l)\mP(l,k)}{Y^2}$.
Substituting into the product rule and combining the first and third terms (which are scalar transposes of each other):
\begin{dmath*}
    \frac{\partial K_{\cal W} ( \mP , {\cal W} ) }{\partial \mW(l,k)} = \frac{\vpi(l)\mP(l,k)}{Y^2}\left[(Y\mathbf{e}_l - \mathbf{n})^\top \mM \vpi_{ \cal W}^\top + \vpi_{ \cal W} \mM (Y\mathbf{e}_l - \mathbf{n})\right] + \vpi_{ \cal W} \frac{\partial \mM}{\partial \mW(l,k)} \vpi_{ \cal W}^\top .
\end{dmath*}
\hfill $\square$

Theorems~\ref{theorem: derivative K^t with pi} and~\ref{thm:derivative_kw} reveal a sharp structural contrast between the two Kemeny constants. The gradient of $K(\bm{P}, \mathcal{W})$ with respect to the weights depends only on the network topology and the stationary distribution of the embedded chain (Theorem~\ref{theorem: derivative K^t with pi}): two networks with identical $\bm{P}$ but vastly different weight structures produce identical gradients, making it impossible for a weight-based optimizer to distinguish a congested edge from a free-flowing one. By contrast, the gradient of $K_{\mathcal{W}}(\bm{P}, \mathcal{W})$ retains full dependence on the weights through $\bm{\pi}_{\mathcal{W}}$, $\bm{M}$, and the normalization quantities $\mathbf{n}$ and $Y$ (Theorem~\ref{thm:derivative_kw}). This sensitivity to the current weight configuration makes $K_{\mathcal{W}}$ the natural objective whenever the optimization acts on the weights themselves.

Together with the first- and second-order gradient expressions established in this section, these results confirm that network objectives built from weighted first-passage times, whether targeting the mean, the variance, or both, are fully amenable to gradient-based optimization. In the next section, we exploit this machinery in two applications.

\section{Applications}\label{sec:app}
We present two applications that highlight the dual role of variance in network optimization. In the surveillance setting of Section~\ref{sub:surveillance}, edge weights represent inspection times (holding-time model), and variance is an \textit{asset}: an unpredictable patrol is harder for an adversary to exploit. In the traffic setting of Section~\ref{sub:traffic}, edge weights represent physical travel times (travel-time model), and variance is a \textit{risk} to be constrained. Both applications rely on the closed-form gradient expressions from the preceding sections, but call for different choices of Kemeny constant following the reflection in Section~\ref{sec:KK}.

\subsection{Variance-Aware Surveillance Policy Optimization}
\label{sub:surveillance}

In many stochastic optimization problems, variance represents risk to be minimized.
The classical mean-variance framework \citep{Markowitz1952} seeks to maximize expected return while minimizing variability, a perspective extended to Markov decision processes through variance-constrained formulations \citep{guo2009mean,xia2020risk}.
However, in adversarial settings such as surveillance, the role of variance is fundamentally different: \emph{unpredictability is an asset, not a liability}.
A low-variance patrol is exploitable, since an intruder who observes the pattern can time intrusions to coincide with predictable coverage gaps \citep{duan2021stochastic}.

Existing Markov chain-based surveillance strategies optimize either for efficiency, by minimizing the (weighted) Kemeny constant \citep{patel2015robotic,franssen2025first}, or for unpredictability, by maximizing the entropy rate \citep{george2018markov,duan2019markov}; see \citet{duan2021markov} for a survey.
However, an intruder cares about the distribution of time until the next patrol visit, not the per-transition randomness of the agent's movements.
Entropy-based criteria quantify the latter rather than the variability in \emph{when} a patroller returns, which is the operationally relevant quantity.
This motivates a variance-based approach that directly measures unpredictability in first-passage times, building on the closed-form expressions developed in Section~\ref{sec:mvwpt} and the optimization framework of \citet{franssen2025first}.

Consider a graph $\hat{\mathcal{G}} = (\mathbb{V}, \mathbb{E})$,
where each node $i \in \mathbb{V}$ represents a local area
(e.g., rooms in a museum or neighborhoods in a city)
and the edge set $\mathbb{E}$ encodes the physical layout.
A single agent patrols the environment following a randomized policy
given by a Markov chain with transition matrix $\mP$ on $\hat{\mathcal{G}}$,
where a deterministic $\mathcal{W}(i,j)$ denotes the time the agent needs
to inspect area~$i$ and travel to area~$j$. Note that due to the deterministic nature of the weights, (A$^*$) is trivially satisfied.
Let $\vmu$ be a target probability vector with $\vmu(i) > 0$ for all $i \in \mathbb{V}$,
representing the desired long-run fraction of time the agent spends in each area.
The target $\vmu$ may reflect surveillance priorities,
historical incident data, or simply a uniform distribution.
In the degenerate case, $\mP$ corresponds to a Hamiltonian cycle,
which minimizes mean passage times on many graphs \citep{franssen2025first} but is maximally predictable: the time and location of the agent can be anticipated with certainty.
To overcome this \emph{curse of predictability},
we introduce variance in the round-trip times of the surveillance agent. Following the discussion in Section~\ref{sec:KK}, this application adopts the holding-time interpretation of edge weights, making $K_{\cal W}$ and $V_{\cal W}$ the appropriate connectivity and variability measures.

We seek a transition matrix $\mP$ on $\hat{\mathcal{G}}$
with stationary distribution $\vmu$
that simultaneously achieves maximal variance
relative to the mean first-passage times.
We define the \emph{surprise index} of a surveillance policy $\mP$, with exogenous inspection times $\cal W$ as:
\begin{equation*}\label{eq:surprise_index}
    \mathcal{S}(\mP, {\cal W}) := \frac{\sqrt{V_{\mathcal{W}}(\mP, {\cal W})}}{K_{\mathcal{W}}(\mP, {\cal W})},
\end{equation*}
where $K_{\mathcal{W}}(\mP, {\cal W})$ is the first-order Kemeny constant (defined via the MWFPTs, see \eqref{eq:fok}) and $V_{\mathcal{W}}(\mP, {\cal W})$ is the second-order Kemeny
(defined via the VWFPTs, see \eqref{eq:sok}).
The surprise index measures the overall network variability in first-passage times
relative to overall network connectivity:
a high $\mathcal{S}$ indicates that unpredictability dominates operational efficiency.
The square-root formulation yields a coefficient-of-variation type metric
that is dimensionless and scale-invariant,
enabling meaningful comparisons across networks with different weight configurations.
As a reference scale:
$\mathcal{S} = 1$ means the standard deviation of weighted first-passage times
equals the mean, so return times are as dispersed as they are long on average;
$\mathcal{S} \ll 1$ indicates a near-deterministic patrol
(e.g., a Hamiltonian cycle, where return times are tightly concentrated);
and $\mathcal{S} > 1$ indicates a policy
where the spread in return times exceeds the average,
making it difficult for an adversary
to predict when any given area will next be visited.

This leads to the following optimization problem,
where we aim to find a surveillance policy $\mP$
that maximizes the surprise index $\mathcal{S}(\mP, \mathcal{W})$,
given exogenous service times $\mathcal{W}$:
\begin{subequations}
\begin{align}
    \max_{\mP} \quad & \mathcal{S}(\mP, \mathcal{W}) \label{eq:objective}\\
    \text{subject to} \quad
    & \vmu^\top \mP = \vmu^\top \label{eq:stationary}\\
    & \mP \mathbf{1} = \mathbf{1} \label{eq:rowstochastic}\\
    & \mP(i,j) \geq \eta, \quad \forall (i,j) \in \mathbb{E} \label{eq:lowerbound}\\
    & \mP(i,j) = 0, \quad \forall (i,j) \notin \mathbb{E}. \label{eq:support}
\end{align}%
\end{subequations}%
Constraint~\eqref{eq:stationary} enforces that the long-run visit frequencies
match the target distribution~$\vmu$;
\eqref{eq:rowstochastic} ensures $\mP$ is row-stochastic;
\eqref{eq:lowerbound} imposes a minimum transition probability
$\eta \in (0, 1/d_{\max})$ on every active edge,
where $d_{\max} = \max_{i \in \mathbb{V}} |\{j : (i,j) \in \mathbb{E}\}|$,
guaranteeing irreducibility;
and~\eqref{eq:support} restricts transitions to the physical layout.
The fractional objective~\eqref{eq:objective} is generally non-convex,
motivating the use of the SPSA-based gradient method
of \citet{franssen2025first}.
The well-posedness of the problem
and the existence of solutions are established below.

\begin{theorem}[Existence of solutions]\label{thm:wellposedness}
Let $\hat{\mathcal{G}} = (\mathbb{V}, \mathbb{E})$ be strongly connected with $n$ nodes, and let $\vmu > 0$ be strictly positive. Then the feasible set
\begin{equation*}
    \mathcal{F} = \left\{\mP \in \mathbb{R}^{n \times n} :
    \mP \text{ satisfies }
    \eqref{eq:stationary}\text{--}\eqref{eq:support}\right\}
\end{equation*}
is non-empty and compact, and the surprise index $\mathcal{S}(\mP, \mathcal{W})$ attains its maximum on $\mathcal{F}$.
\end{theorem}

\noindent {\bf Proof:}
Let $ \mH$ be the matrix with all elements equal to $ 1 /n $.
Then $ \vmu \mH = \vmu$, and $ \mH \in {\cal F} $
provided that $ 1/ n \geq \eta.$
As for compactness of $ \cal F$, note that the set $\mathcal{F}$ is bounded ($\mP(i,j) \in [\eta, 1]$) and closed, since all constraints are preserved under limits. The lower bound $\mP(i,j) \geq \eta > 0$ on the strongly connected graph $\hat{\mathcal{G}}$ ensures that every limit point is irreducible with a unique stationary distribution $\vmu$. By Heine-Borel, $\mathcal{F}$ is compact.
Finally, the surprise index $\mathcal{S}(\mP, \mathcal{W})$ is continuous on $\mathcal{F}$ as $K_{\mathcal{W}}$ and $V_{\mathcal{W}}$ are both continuous functions of $\mP$ (as compositions of the stationary distribution, the fundamental matrix, and the weight matrices), and $K_{\mathcal{W}}(\mP) > 0$ for all $\mP \in \mathcal{F}$ since every feasible policy is irreducible. By the Weierstrass theorem, $\mathcal{S}$ attains its supremum at some $\mP^* \in \mathcal{F}$, and $\mathcal{S}(\mP^*, \mathcal{W}) < \infty$ since $\mathcal{S}$ is continuous on a compact set. 
{\qed}

Since $\mathcal{F}$ is compact and $\mathcal{S}$ is differentiable in the interior of $\mathcal{F}$, the optimal policy $\mP^*$ is either a stationary point of~$\mathcal{S}$ in the relative interior of $\mathcal{F}$ or lies on the boundary.
In both cases, $\mP^*$ can be approached by gradient-based methods. The SPSA algorithm of \citet{franssen2025first}, operating in the null space of the constraint matrix to maintain feasibility at every iteration, provides a practical method for finding such solutions. 

We illustrate the framework on a $4 \times 4$ grid network
with $n = 16$ nodes and $48$ directed edges.
Edge weights $\mathcal{W}(i,j)$ represent inspection plus travel times,
with uniform target distribution $\vmu = \tfrac{1}{16}\mathbf{1}$
and minimum transition probability $\eta = 10^{-4}$. We first solve the optimization problem with deterministic patrol times (or weights),
such that $\mathcal{W}^{(2)}(i,j) = (\mathcal{W}(i,j))^2$ for all edges,
and the coefficient of variation $\mathrm{CV}(i,j) := \sigma(i,j) / \mathcal{W}(i,j) = 0$.
In this setting, all variance in the first-passage times
arises solely from the stochastic routing decisions encoded in~$\mP$.
The baseline achieves $\mathcal{S} = 1.045$, meaning the standard deviation of weighted first-passage times approximately equals the mean, a natural reference point.
Minimizing variance produces a policy with $\mathcal{S} = 0.214$, a $79.5\%$ reduction, by concentrating transitions along a Hamiltonian cycle: all $16$ nodes have exactly one dominant outgoing edge ($\mP(i,j) > 0.5$), creating a deterministic circuit (Figure~\ref{fig:4x4_policies}(a)). This confirms that low-$\mathcal{S}$ policies are maximally predictable.
The maximum-surprise policy (Figure~\ref{fig:4x4_policies}(b)) achieves $\mathcal{S}^* = 1.408$, a $35\%$ increase over the baseline.
The optimized policy distributes transition probabilities broadly across available edges,
so that an intruder cannot reliably predict the agent's next move.

\begin{figure}
    \centering
    \subfloat[Min-variance ($\mathcal{S} = 0.21$)]{\includegraphics[width=.4\textwidth]{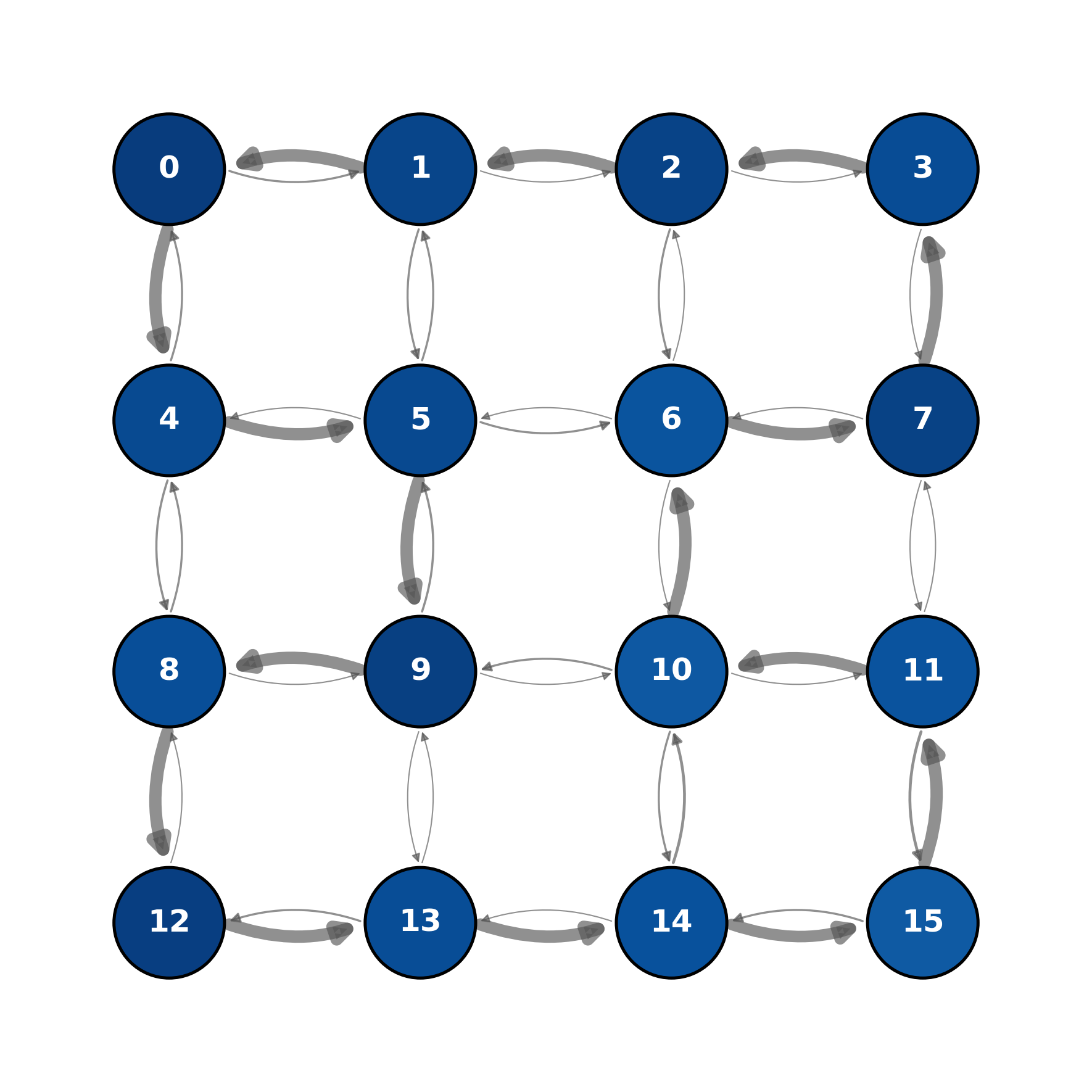}}
    \hspace{0.8cm}
    \subfloat[Max-surprise ($\mathcal{S}^* = 1.41$)]{\includegraphics[width=.4\textwidth]{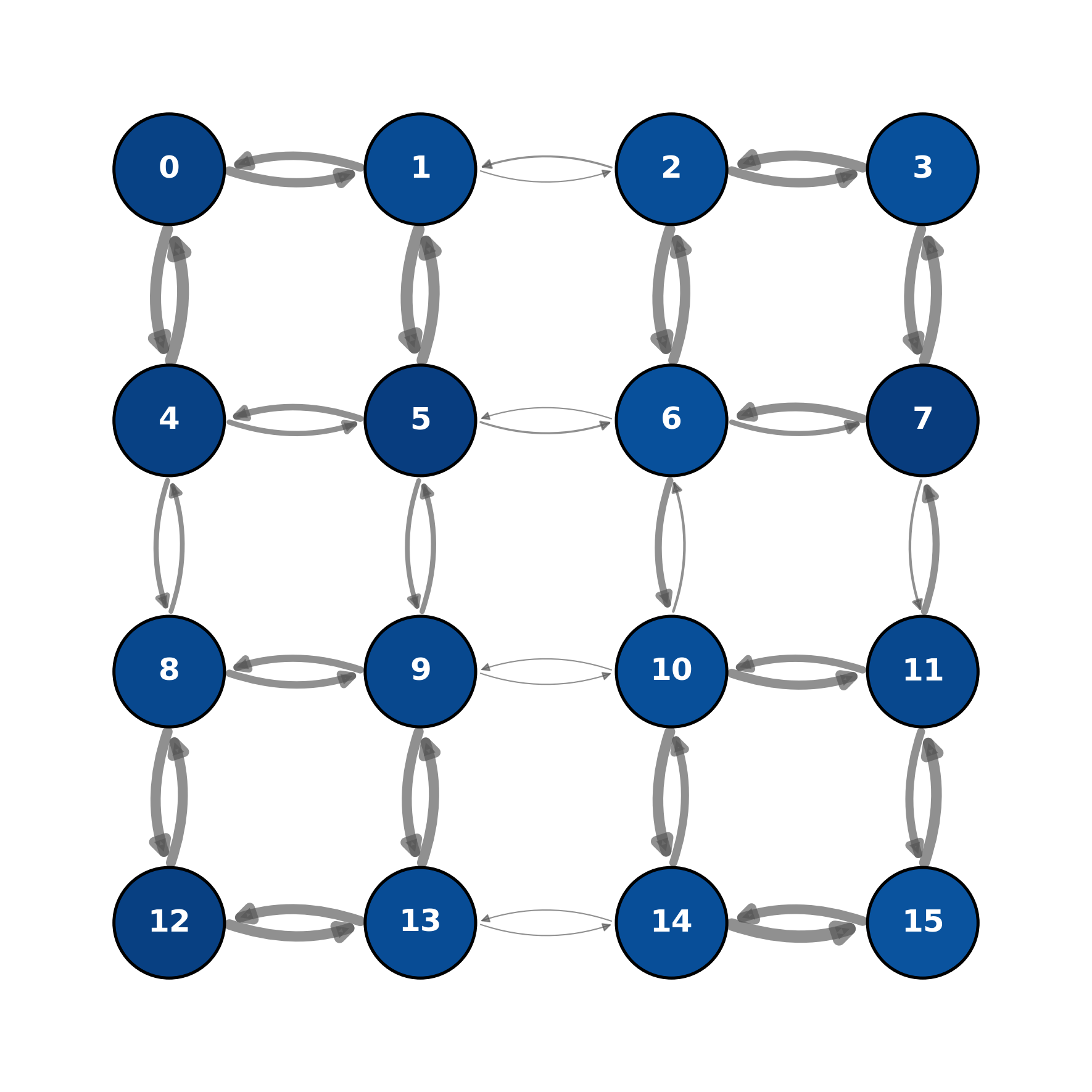}}
\caption{Optimized policies on the $4 \times 4$ grid with deterministic weights.
    (a)~Minimizing variance produces a predictable Hamiltonian cycle
    with $16$ dominant edges ($\mP(i,j) > 0.5$),
    reducing $\mathcal{S}$ by $79.5\%$.
    (b)~Maximizing $\mathcal{S}$ distributes transition probabilities broadly,
    producing high variance in first-passage times.
    Edge thickness indicates $\mP(i,j)$;
    node shading indicates $\boldsymbol{\pi}_{\mathcal{W}}(i)$.}\label{fig:4x4_policies}
\end{figure}

To demonstrate that the framework scales and accommodates non-uniform coverage requirements, we consider an $8 \times 8$ grid with four obstacle nodes arranged as two horizontal pairs
at positions $(2,2)$-$(2,3)$ and $(5,4)$-$(5,5)$, yielding $60$ accessible nodes and $196$ directed edges (Figure~\ref{fig:8x8_results}).
The $12$ nodes adjacent to obstacle clusters receive twice the coverage weight of regular nodes.
Note that not all obstacle configurations satisfy constraint~\eqref{eq:stationary}.
Depending on the target distribution $\vmu$ and the graph topology, replacing nodes by obstacles can alter the degree structure of neighboring nodes in ways that render
the flow-balance constraint $\sum_{i:(i,j)\in\mathbb{E}} \vmu(i)\,\mP(i,j) = \vmu(j)$
infeasible. We verify the feasibility of our instances via linear
programming before running the SPSA algorithm.

With stochastic weights (i.e., the coefficient of variation $\mathrm{CV}$ ranging from $0.30$ to $1.70$, mean $\mathrm{CV} = 0.85$, $40\%$ high-variability edges), the optimized policy achieves $\mathcal{S}^* = 1.255$, a $16\%$ increase over the baseline ($\mathcal{S} = 1.080$).
The more conservative gains on the $8 \times 8$ grid reflect the tighter feasible set imposed by the non-uniform stationary constraint and the larger network: the hard constraint~\eqref{eq:stationary} leaves less freedom to redistribute transition probabilities while simultaneously respecting the $2{:}1$ priority ratio. The correlation between optimized transition probabilities and edge CVs is $\rho(\mP, \mathrm{CV}) = -0.055$, indicating that the policy does not preferentially select high-CV edges; unpredictability is achieved primarily through the redistribution of routing probabilities rather than through exploitation of weight variability.

\begin{figure}
    \centering
    \includegraphics[width=.75\textwidth]{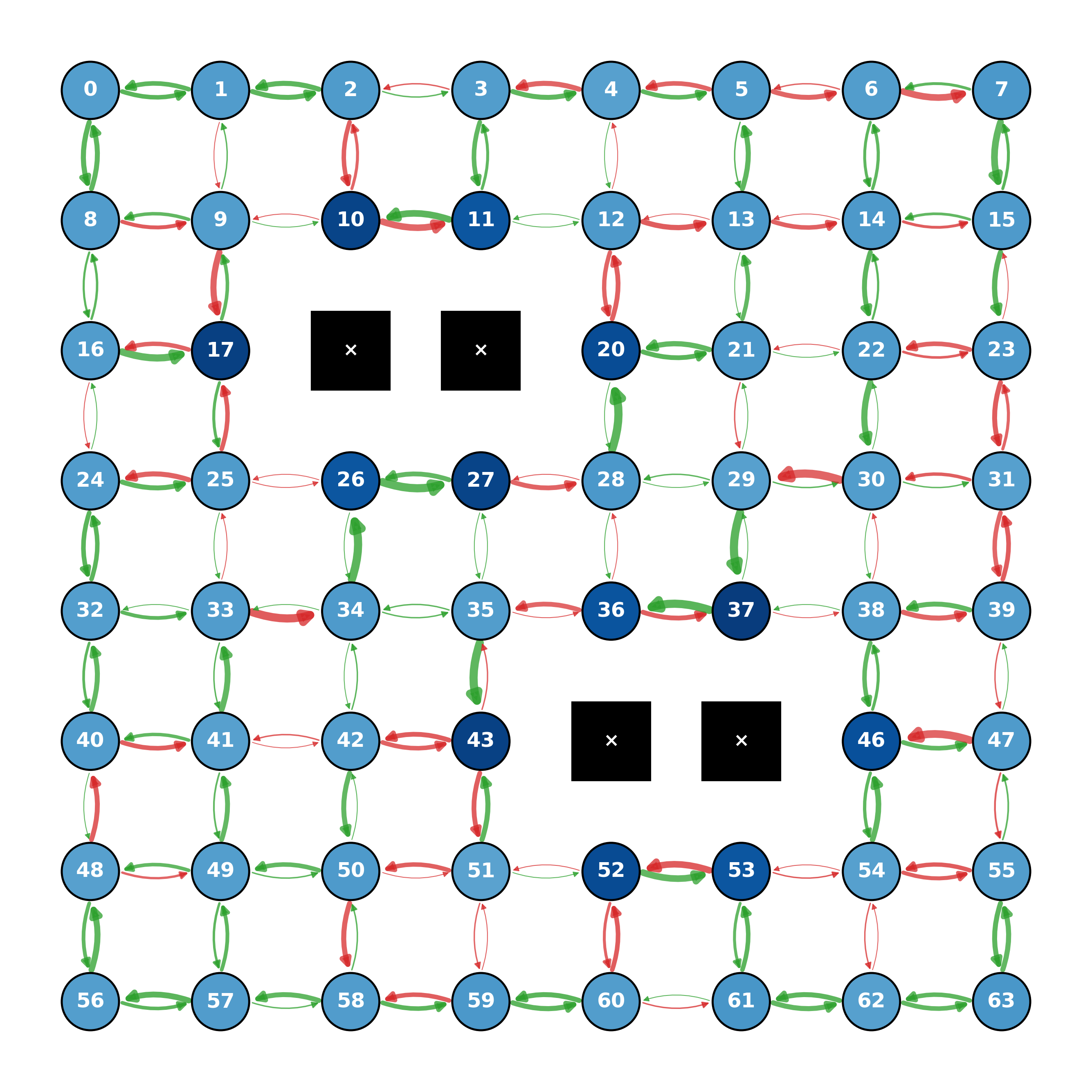}
\caption{Maximum-surprise policy on an $8 \times 8$ grid
    with non-uniform $\vmu$ and stochastic weights ($\mathcal{S}^* = 1.26$).
    Black squares denote obstacles;
    darker nodes indicate higher target coverage weight.
    Green edges have low CV ($< 1$),
    red edges have high CV ($\geq 1$);
    edge thickness indicates $\mP(i,j)$.
    The policy respects the $2{:}1$ priority ratio
    on obstacle-adjacent nodes.}\label{fig:8x8_results}
\end{figure}

Table~\ref{tab:surveillance_summary} consolidates the results. The experiments confirm that maximizing $\mathcal{S}$ consistently yields policies that exceed the baseline across grid sizes,
weight configurations, and coverage requirements, while minimizing variance produces maximally predictable Hamiltonian cycles ($\mathcal{S} \approx 0.21$).

Decomposing $\mathcal{S}$ into its components reveals that the optimization achieves surprise through qualitatively different mechanisms depending on the network.
On the unconstrained $4 \times 4$ grid, both $K_{\mathcal{W}}$ and $\sqrt{V_{\mathcal{W}}}$ increase (from $28.69$ to $2{,}622$ and from $29.98$ to $3{,}691$, respectively): the policy creates unpredictability by allowing long, indirect patrol routes that inflate both the mean and the spread of return times, with the spread growing faster than the mean.
On the constrained $8 \times 8$ grid, both $K_{\mathcal{W}}$ and $\sqrt{V_{\mathcal{W}}}$ \emph{decrease} (from $262.87$ to $202.58$ and from $283.95$ to $254.16$): here the non-uniform stationary constraint limits the ability to create long detours, so the policy instead achieves surprise by \emph{tightening connectivity}, or reducing mean passage times while preserving
a disproportionate share of the variance.
In both cases, $\mathcal{S}$ increases, but through opposite movements of its components.
This shows that the surprise index does not prescribe a single structural pattern, but rather adapts to the feasible set of each network, finding the best achievable variance-to-mean ratio
within the constraints imposed by the graph topology and the target distribution.

The framework naturally accommodates stochastic edge weights and non-uniform targets. These features are required for realistic surveillance networks. However, they are not naturally handled by existing entropy-based approaches \citep{george2018markov,duan2019markov}, which optimize over doubly stochastic matrices and implicitly assume uniform coverage.

\begin{table}[H]
\centering
\caption{Performance of different optimized surveillance policies for two network instances. The table reports $K_{\mathcal{W}}$, $\sqrt{V_{\mathcal{W}}}$, the surprise index $\cal S$, and the gain of the maximum-surprise policy compared to the baseline.}
\label{tab:surveillance_summary}
\small
\begin{tabular}{llrrrr}
\toprule
\textbf{Instance} & \textbf{Policy} & $\qquad \qquad K_{\mathcal{W}}$ & $\qquad \qquad \sqrt{V_{\mathcal{W}}}$ & $\qquad$ \textbf{Surprise index} $\mathcal{S}$ & $\qquad$ \textbf{Gain} \\ \midrule
 $4 \times 4$ grid   & Min-Variance    & 9.57                            & 2.05                                   & 0.214                                                           & --                                      \\
             & Baseline        & 28.69                           & 29.98                                  & 1.045                                                           & --                                      \\
                  & Max-Surprise    & 2,622.00                    & 3,691.00                           & 1.408                                                           & $+35\%$                                 \\ \hline
$8 \times 8$ grid & Baseline        & 262.87                          & 283.95                                 & 1.080                                                           & --                                      \\
                  & Max-Surprise    & 202.58                          & 254.16                                 & 1.255                                                           & $+16\%$                                 \\ \bottomrule
\end{tabular}%
\end{table}

\subsection{Policy-Aware Traffic Optimization Under Sequential Failures}
\label{sub:traffic}
In this subsection, we apply the traversal-time interpretation of our model (see Section~\ref{sec:ime}) to analyze the vulnerability of road networks under sequential failures. This application is motivated by foundational studies on network robustness \citep{Albert_Jeong_Barabasi_2000} and practical link-removal methodologies developed specifically to measure disruption impacts in traffic systems \citep{Jenelius_Petersen_Mattsson_2006}. Furthermore, we build upon existing literature that leverages the standard Kemeny constant for network regulation and control \citep{kirkland2010fastest, Crisostomi_Kirkland_Shorten_2011}, extending this global connectivity metric to our fully weighted Markovian graph framework.

Let the graph ${\cal G} =(\sV, \sE)$ represent a road network where nodes $\sV$ are intersections and edges $\sE$ are road segments. The elements $\mP(i,j)$ of the transition matrix model the routing choices of drivers, leading to a stationary occupancy measure $\vpi^*$ that represents the equilibrium traffic flow across the network. For each edge $(i,j) \in \sE$, the deterministic weight $\mW(i,j)$ models the physical travel time, which is controlled by the operator and inversely related to the speed limit. Because these traversal times are deterministic, they again trivially satisfy the i.i.d. requirement of (A$^*$). In this framework, the routing matrix $\mP$ can be viewed as the decentralized ``solution'' generated by drivers to achieve the occupancy measure $\vpi^*$ under the current speed limit regime $\mW$. Ultimately, the overall quality and global efficiency of this resulting traffic distribution are quantified by the weighted Kemeny constant. Following the discussion in Section~\ref{sec:KK}, the travel-time interpretation of the traffic model calls for $K(\mP, {\cal W})$, introduced in \eqref{travel_Kemeny_def}, which we adopt in this application.

Assume that a failure of a road element occurs, rendering a specific edge $(i,j)$ unusable; for example, an extreme weather event, an accident, or an obstruction on a road, that makes it completely unusable. The loss of a road alters the network topology, forcing the traffic flow $\mP(i,j)$ to zero (note that this is effectively not the same as setting $\mW(i,j)$ to zero, as the travel time model would interpret this as an instantaneous transition). We then distinguish the following policies corresponding to the drivers' reaction to the road failure:
\begin{enumerate}
    \item {\bf Unsupervised policy:} 
    If $ ( i , j ) $ fails, the traffic flow is redistributed to the other outgoing arcs of $ i$, that is, we set $ \mP ( i , j ) = 0$ and obtain $ \tilde \mP$ through row-normalization, when possible, and otherwise there is no feasible solution. Note that this policy does not enforce the occupancy measure 
    $\vpi^*$.

    \item {\bf Supervised policy:} The traffic flow after failure preserves the original stationary distribution $\vpi^*$, that is, we constrain $ \tilde \mP $ in our optimization so that $ \tilde \mP$ has stationary distribution $\vpi^*$, where we apply the method from \cite{franssen2025first}.
    
    \item {\bf Locally supervised policy:} We partition the node set $\sV$ into two disjoint subsets $\cal D$ and $\cal T$ such that $\mathcal{D} \cup \mathcal{T} = \sV$ and $\mathcal{D} \cap \mathcal{T} = \emptyset$. Here, the set $\cal D$ represents the set of important \textit{destinations}, and we aim to uphold the same long-term visit fractions as in the original network. On the other hand, we let $\cal T$ denote \textit{transit} nodes, corresponding to simple traffic intersections. This approach reflects the fact that drivers may be indifferent about what intersections they have to cross to reach their destination, the long-term fraction of visits can be more elastic.
\end{enumerate}

\begin{remark}
Deviating from the machine learning literature, we use  the term ``supervision'' to express the dependence of the
local decisions of the drivers on the target distribution $\vpi^*$. 
An ``unsupervised'' policy is allowed to organize the traffic network flow without regard of the target distribution.
\end{remark}

After the readjustment of traffic flow, the policymaker has the ability to control the traversal time matrix $ \mW$ by systematically altering speed limits to mitigate the impact of the failure.
This leads to the following problem: find a new speed limit matrix $ \tilde \mW$, so that the drivers flow that wants to realize occupancy $ \vpi^\ast $ changes to an adjusted traffic flow $ \tilde \mP$, with the constraint that the overall connectivity does not worsen, which is expressed by $ K ( \tilde \mP , \tilde {\cal W} )$ and $ V ( \tilde \mP , \tilde {\cal W} )$ not deteriorating, where $ V ( \tilde \mP , \tilde {\cal W} )$ denotes the second-order Kemeny constant from \eqref{sec-order Kemeny}.

The central problem for the traffic authority is to determine how to adjust the traversal times $\mW$ (via speed-limit regulations) on the \textit{remaining} operational links to mitigate the disruption's impact. Therefore, we choose to formulate the recovery as a \textit{constrained minimal intervention} problem,
\begin{subequations}
\begin{align}
    \min_{\mW} \quad &|| {\tilde \mW} - \mW ||^2\label{tn: objective}\\
    \text{subject to} \quad 
    & K (\mP, {\cal W})-K({\tilde \mP}, {\tilde {\cal W}}) \geq 0\label{tn: kemeny constraint}\, ,\\
    & \mW_{\rm MIN} \leq \mW \leq \mW_{\rm MAX}\label{tn: min-max constraint}\, ,\\
    & V(\mP, {\cal W})-V({\tilde \mP}, {\tilde {\cal W}})\geq 0\, , \label{tn: variance constraint}
\end{align}%
\end{subequations}%
where \eqref{tn: kemeny constraint} ensures the connectivity of the perturbed network is \textit{not worse} than the original network, \eqref{tn: min-max constraint} represents the physical bounds on the weights, and \eqref{tn: variance constraint} represents a variability constraint ensuring that the variance of travel times does not increase.

The optimization problem \eqref{tn: objective}-\eqref{tn: variance constraint} has a strictly convex objective and partially convex constraints. The Kemeny constraint \eqref{tn: kemeny constraint} is convex since \eqref{travel_Kemeny_def} is affine in $\mW$, and the box constraints~\eqref{tn: min-max constraint} are convex. However, the variance constraint~\eqref{tn: variance constraint} is non-convex, as numerical verification has shown that \eqref{sec-order Kemeny} is not a convex function of $\mW$, even for deterministic $\mW$. Therefore, \eqref{tn: objective}-\eqref{tn: variance constraint} is only guaranteed to find local optima.

For the \textit{minimal intervention} framework, we assume local optimality is sufficient, since our goal is to find \textit{small} feasible adjustments that preserve network performance, not necessarily the globally smallest such adjustment. The sequential nature of the framework ensures we track a continuous, physically meaningful trajectory through the weight space. Moreover, the distinction between local and global optima becomes negligible when the feasible set is small, as occurs when the network approaches its ``breakdown'' point.

The following lemma shows that a locally supervised policy can always be constructed.
\begin{lemma}[Locally supervised policy]\label{le:lsp}
For edge failure $ ( i , j )$, let $ \mP'$ denote the normalized version of $ \mP$ with entry $ ( i , j )$ set to zero, and denote the stationary distribution of $ \mP'$ by $ \vpi '$. Under the locally supervised policy, the adjusted occupancy measure is given by
\begin{equation}\label{hybrid_stat_dist}
    \widehat \vpi (i) := \begin{cases}
        \vpi^* (i) \;\;\; &\text{if } i \in {\cal D}\\
        \alpha \cdot \vpi' (i) \;\;\; &\text{if } i \in {\cal T}\, ,
    \end{cases}
\end{equation}
where $\alpha$ is a scaling factor required to ensure that $\sum_i \widehat \vpi(i) = 1$.
\end{lemma}

{\bf Proof:}
Since the probability mass allocated to the destinations is kept fixed at $M_{\cal D} := \sum_{j \in {\cal D}} \vpi^*(j)$, the remaining probability mass available for the transit nodes must be $1 - M_{\cal D}$. With this insight, the scaling factor can easily be derived using simple algebra,
 $   \alpha = (1-M_{\cal D})/\sum_{j \in {\cal T}} \vpi' (j) $.
\hfill$ \square $

The construction in Lemma~\ref{le:lsp} ensures that the stationary distribution for all destination nodes is always preserved, while for the transit nodes, the probability distribution is normalized and rescaled to account for the other fixed probabilities. After constructing the hybrid stationary distribution $\widehat \vpi$, we can simply employ the projection algorithm described by \cite{franssen2025first} to find the transition matrix $\widehat \mP$ such that ${\widehat \vpi}{\widehat \mP} = {\widehat \vpi}$.

It is important to note that the failure of a critical link, or a sequence of links in a cascading failure scenario, may render the feasible set for \eqref{tn: objective}-\eqref{tn: variance constraint} empty. This occurs if the topological damage is so severe that no valid assignment of weights can restore the Kemeny constant or variance to within the acceptable bounds.

In the following, we present a simulation study of the different policies, where
we successively remove links, and adjust the speed limit matrix after each link removal.
For simulation, we construct a synthetic graph topology using geometric random graphs, as proposed in \cite{Penrose_2003}. Each edge in the graph $(i,j) \in \mathbb{E}$ is assigned an initial travel time $\mW(i,j)$ sampled from a uniform distribution between predefined minimum and maximum bounds. The lower bound was established at $2/3$ of the original weight, while the upper bound was set at $4/3$. We proceed by iteratively removing a random edge, simulating an unpredictable disruption. Before proceeding, we verify that the removal does not decompose the graph into disconnected components, as this would indicate that no traffic flow is possible between two components. Once the edge is removed, we employ one of the three policies, obtaining $\tilde \mP$, and optimize the weights until either projecting or optimizing becomes infeasible, resulting in $\tilde \mW$. During each iteration, we record the resulting Kemeny constant $K(\tilde \mP, \tilde{\cal W})$ and the second-order Kemeny constant $V(\tilde \mP, \tilde {\cal W})$. These metrics are then compared with an "unoptimized" counterpart, $K(\tilde \mP, {\cal W})$ and $V(\tilde \mP, {\cal W})$, serving as a baseline in which the topology is damaged and $\mW$ is not updated.
\begin{figure}[t]
    \centering
    \subfloat[]{\includegraphics[width=\linewidth]{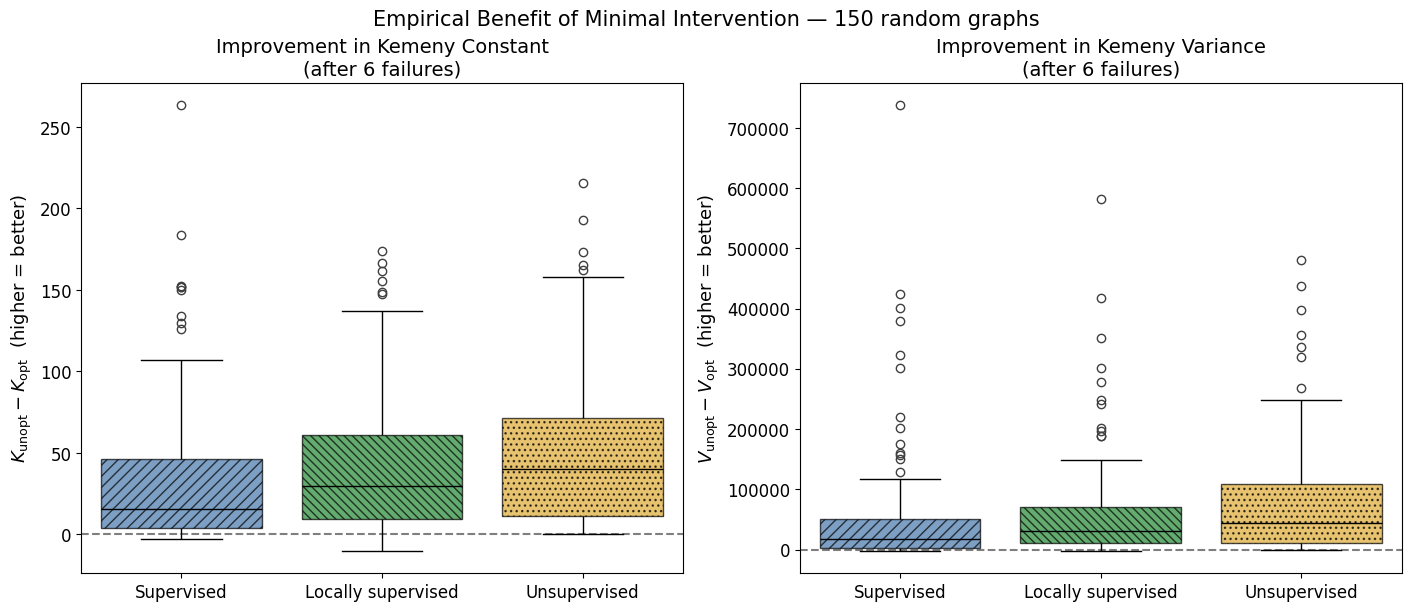}}\\
    \subfloat[]{\includegraphics[width=0.85\linewidth]{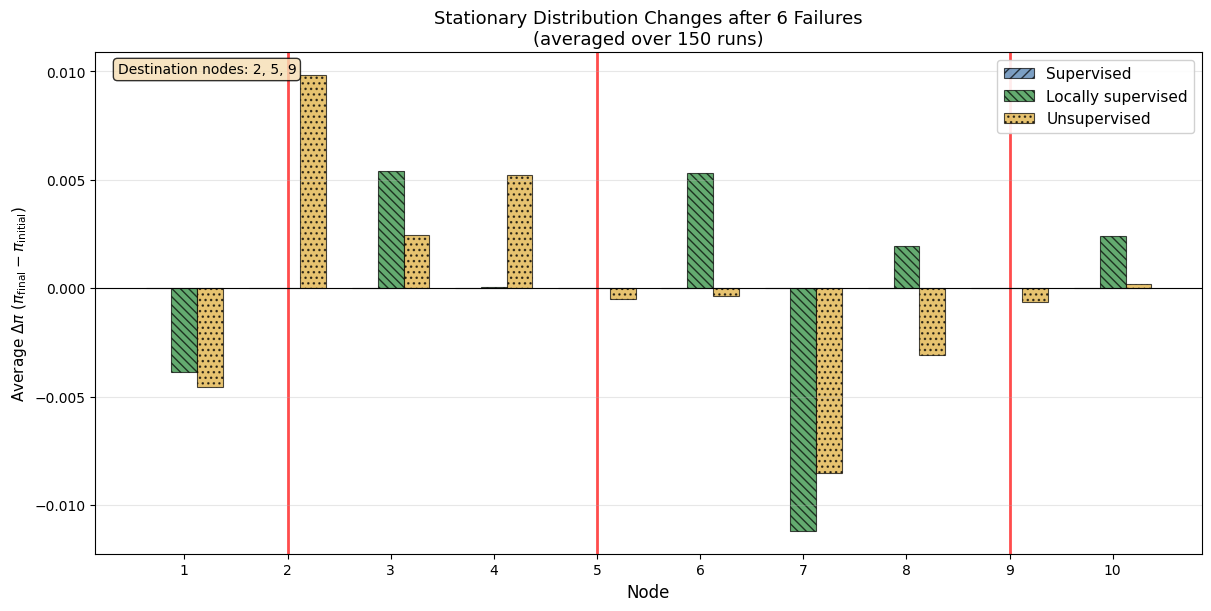}}
    \caption{(a) Aggregated improvement in the Weighted Kemeny Constant (left) and Variance (right) for 150 randomly generated geometric graphs with 10 nodes, $k=5$, and connection distances $r = \sqrt{d/(10 \pi)}$, where $d=5$ is the desired average degree. Nodes 2, 5, and 9 were arbitrarily chosen to represent destinations. (b) Comparison of the average change to the stationary distribution for all three policies.}
    \label{fig:boxplot_3way}
\end{figure}

\begin{table}[t]
\centering
\caption{Summary statistics for 150 cascading failure simulations comparing the three policies. 
}
\label{tab:summary_kemeny}
\small
\begin{tabular}{lrrr}
\toprule
\textbf{Policy} & \textbf{Supervised} & \textbf{Unsupervised} & \textbf{Locally supervised} \\
\midrule
{Successful runs}  & 109        & 123        & 88        \\
{Mean $\Delta K(\mP, {\cal W})$}   & 33.58   & 49.79    & 45.96    \\
{Mean $\Delta V(\mP, {\cal W})$} & 54,176.57 & 77,574.79 & 67,777.89 \\
{Mean $\Delta \vpi$} & 0.00 & 0.03 & 0.02\\
{Max $\Delta \vpi(i)$ for $i \in \cal D$} & 0.00 & 0.17 & 0.00\\
\bottomrule
\end{tabular}
\end{table}

Before analyzing the global efficiency, we must first verify that the regulatory policies behave as intended. Figure~\ref{fig:boxplot_3way} (b) illustrates the average change in the stationary distribution after the sequential failures. The unconstrained policy completely abandons the original occupancy measure, allowing (relatively) massive deviations in traffic flow across all nodes. In stark contrast, both the supervised and locally supervised policies perfectly preserve the exact pre-failure traffic flow at the critical destination nodes (in this case nodes 2, 5, and 9), exhibiting zero deviation. Moreover, the supervised policy shows no deviation in \textit{any} node, as expected. This confirms that our projection algorithm successfully forces the network to maintain vital service guarantees.

However, Figure~\ref{fig:boxplot_3way} (a) reveals the severe systemic cost of enforcing these strict service guarantees in terms of the global connectivity. The supervised policy severely restricts the feasible space for optimization, yielding the lowest mean $\Delta K(\mP, {\cal W})$ and even resulting in negative performance (whiskers dipping below zero) in certain topologies. On the other hand, while the unsupervised policy achieves massive improvements by allowing the network to naturally reorganize, it does so at the unacceptable cost of destroying the original service levels at the important destinations.

This reveals the key advantage of the locally supervised policy. By relaxing flow constraints purely on transit intersections while strictly protecting the service levels at the destination nodes, the optimizer reclaims a massive amount of the network's natural redundancy. As seen in Table~\ref{tab:summary_kemeny}, the locally supervised approach achieves a median $\Delta K(\mP, {\cal W})$ that vastly outperforms the supervised policy and approaches the theoretical upper bound set by the unsupervised policy. It represents a Pareto-optimal compromise: we successfully shield critical infrastructure from the effects of topological damage, while exploiting the elasticity of transit routes to maximize global network efficiency.

While the locally supervised policy achieves a Pareto-optimal trade-off when successful, the data also reveals a noticeably lower success rate compared to both other policies, as can be observed in Table~\ref{tab:summary_kemeny}. At first glance, this appears counterintuitive that a \textit{partially} relaxed policy would yield infeasible solutions more frequently than a strictly contained one. However, this phenomenon stems directly from the construction of the hybrid distribution in Lemma~\ref{le:lsp}. By artificially mixing the original pre-failure stationary probabilities, $\vpi^*$, and the post-failure naive ones, $\alpha \cdot \vpi'$, the policy demands a very specific corresponding transition matrix. In severely degraded, sparse topologies, the remaining edges simply lack the routing capacity to bridge these two differing regimes into one transition matrix. Consequently, the projection algorithm fails to find a suitable transition matrix, highlighting that hybrid resilience strategies are strictly bounded by the physical routing capacity. When increasing the average degree $d$ of the graph, the number of successful runs eventually increases towards the number of simulations for all three policies.

The trade-off is clear: local supervision is the most selective method, because it imposes partial constraints that must be satisfied exactly. However, when feasible, it delivers substantial optimization gains while preserving service guarantees at the designed destinations. The results indicate that the ``cost'' of fixing the stationary probability at important nodes is relatively low, provided that the network is still allowed to reorganize flow through the transit nodes. By relaxing constraints on the transit layer while pinning the service layer, the locally supervised policy recovers most of the optimization potential lost when fully preserving $\vpi^*$. The approach thus represents a principled balance: preserve what matters, release what does not, and exploit the resulting flexibility.

\section{Conclusion and Discussion}\label{sec:Conclusion}
In this paper, we introduced a random walk model that includes additional edge weights representing the cost of traversing an edge, decoupling the network's topological routing behavior from its physical travel costs. We obtained closed-form expressions for both the MWFPT and the VWFPT, as well as their partial derivatives with respect to our parameters of interest. A deliberate choice in these derivations is the use of the fundamental matrix as a generalized inverse, yielding an expression for the MWFPT that is linear in the weights $\mW$, which is not guaranteed for alternative generalized inverses. A primary theoretical contribution is the formulation of a novel weighted Kemeny constant and its natural extension to a second-order Kemeny constant, providing an additional metric for global network efficiency in terms of variance. We demonstrated the versatility of this model through two applications: optimizing patrol unpredictability in surveillance networks and preserving connectivity under cascading failures in traffic networks.

Several limitations and practical considerations should be noted. First, our variance results require the i.i.d.\ assumption (A$^*$) for the weight matrix, which excludes settings with temporally correlated travel times. Second, our model assumes a single random walker; extending to multiple interacting agents would require joint first-passage distributions, substantially increasing computational complexity. Third, the optimization problems we consider are generally non-convex and we rely on local optimization via SPSA. While local optimality suffices for the minimal-intervention framework, where the search targets small feasible adjustments, it leaves open the question of global optimality in general. On the practical side, both applications employ the null-space projection method of \citet{franssen2025first} for gradient-based optimization. While effective, the null-space perturbation vectors grow with the size of the transition matrix, which can lead to significant memory requirements.

A natural extension is to consider higher-order moments of weighted first-passage times, enabling risk-sensitive optimization beyond the mean-variance paradigm. Incorporating time-dependent weights would bring the model closer to real-world applications, though likely at the cost of closed-form tractability. The connection between the surprise index and game-theoretic patrolling formulations also deserves further exploration: a fully adversarial model of the intruder's best response could lead to Stackelberg equilibrium policies that integrate naturally with the variance-based framework developed here.

\section*{Acknowledgments}
A.~Zocca and R.~van der Burg are partially supported by the NWO VI.Vidi.233.247 grant \textit{Power Network Optimization in the Age of Climate Extremes} \url{https://doi.org/10.61686/GOOEL09973}.

\section*{Code and Data Disclosure}\label{sec:Code and Data Disclosure}The code and data to support the numerical experiments in this paper can be found at \url{https://github.com/thaoyuh/Weighted_Markovian_Graphs}.

\bibliographystyle{unsrtnat}
\bibliography{sample}  

@article{von2007tutorial,
  title={A tutorial on spectral clustering},
  author={Von Luxburg, Ulrike},
  journal={Statistics and computing},
  volume={17},
  number={4},
  pages={395--416},
  year={2007},
  publisher={Springer}
}

@article{masuda2017random,
  title={Random walks and diffusion on networks},
  author={Masuda, Naoki and Porter, Mason A and Lambiotte, Renaud},
  journal={Physics reports},
  volume={716},
  pages={1--58},
  year={2017},
  publisher={Elsevier}
}

@article{pearson,
  title={The problem of the random walk},
  author={Pearson, Karl},
  journal={Nature},
  volume={72},
  number={1867},
  pages={342--342},
  year={1905},
  publisher={Nature Publishing Group UK London}
}

@book{lawler2010random,
  title={Random walk: a modern introduction},
  author={Lawler, Gregory F and Limic, Vlada},
  volume={123},
  year={2010},
  publisher={Cambridge University Press}
}

@article{zhang2013random,
  title={Random walks on weighted networks},
  author={Zhang, Zhongzhi and Shan, Tong and Chen, Guanrong},
  journal={Physical Review E—Statistical, Nonlinear, and Soft Matter Physics},
  volume={87},
  number={1},
  pages={012112},
  year={2013},
  publisher={APS}
}

@article{newman2004finding,
  title={Finding and evaluating community structure in networks},
  author={Newman, Mark EJ and Girvan, Michelle},
  journal={Physical review E},
  volume={69},
  number={2},
  pages={026113},
  year={2004},
  publisher={APS}
}

@book{borkar2012hamiltonian,
  title={{H}amiltonian cycle problem and {M}arkov chains},
  author={Borkar, Vivek S and Ejov, Vladimir and Filar, Jerzy A and Nguyen, Giang T},
  volume={171},
  year={2012},
  publisher={Springer Science \& Business Media}
}

@incollection{sarkar2011random,
  title={Random walks in social networks and their applications: a survey},
  author={Sarkar, Purnamrita and Moore, Andrew W},
  booktitle={Social Network Data Analytics},
  pages={43--77},
  year={2011},
  publisher={Springer}
}

@book{jackson,
  title = {Social and Economic Networks},
  publisher = {Princeton University Press},
  author = {Jackson,  Matthew O.},
  year = {2010},
  month = nov 
}

@techreport{page1999pagerank,
  author      = {Page, Lawrence and Brin, Sergey and Motwani, Rajeev and Winograd, Terry},
  title       = {The PageRank Citation Ranking: Bringing Order to the Web},
  institution = {Stanford InfoLab},
  year        = {1999},
  type        = {Technical Report}
  }

@article{fogaras2005towards,
  author={Fogaras, D{\'a}niel and R{\'a}cz, Bal{\'a}zs and Csalog{\'a}ny, K{\'a}roly and Sarl{\'o}s, Tam{\'a}s},
  title   = {Towards Scaling Fully Personalized {PageRank}: {A}lgorithms, Lower Bounds, and Experiments},
  journal = {Internet Mathematics},
  volume  = {2},
  pages   = {333--358},
  year    = {2005}
  }

@article{haveliwala2003topic,
  author  = {Haveliwala, Taher H.},
  title   = {Topic-Sensitive {PageRank}: {A} Context-Sensitive Ranking Algorithm for Web Search},
  journal = {IEEE Transactions on Knowledge and Data Engineering},
  volume  = {15},
  pages   = {784--796},
  year    = {2003}
  }

@inproceedings{pan2004automatic,
  author    = {Pan, Jia-Yu and Yang, Hyung-Jeong and Faloutsos, Christos and Duygulu, Pinar},
  title     = {Automatic Multimedia Cross-Modal Correlation Discovery},
  booktitle = {Proceedings of the 10th ACM SIGKDD International Conference on Knowledge Discovery and Data Mining},
  year      = {2004},
  pages     = {653--658}
}

@article{shen2014lazy,
  author  = {Shen, Jianbing and Du, Yunfan and Wang, Wenguan and Li, Xuelong},
  title   = {Lazy Random Walks for Superpixel Segmentation},
  journal = {IEEE Transactions on Image Processing},
  volume  = {23},
  number  = {4},
  pages   = {1451--1462},
  year    = {2014},
  month   = {April}
}

@article{mccrea1936problem,
  title={A problem on random paths},
  author={McCrea, WH},
  journal={The Mathematical Gazette},
  volume={20},
  number={241},
  pages={311--317},
  year={1936},
  publisher={JSTOR}
}

@article{hunter2018computation,
  title={The computation of the mean first passage times for {M}arkov chains},
  author={Hunter, Jeffrey J},
  journal={Linear Algebra and its Applications},
  volume={549},
  pages={100--122},
  year={2018},
  publisher={Elsevier}
}

@article{ballal2022network,
  title={Network community detection and clustering with random walks},
  author={Ballal, Aditya and Kion-Crosby, Willow B and Morozov, Alexandre V},
  journal={Physical Review Research},
  volume={4},
  number={4},
  pages={043117},
  year={2022},
  publisher={APS}
}

@article{patel2015robotic,
  title={Robotic surveillance and Markov chains with minimal weighted {K}emeny constant},
  author={Patel, Rushabh and Agharkar, Pushkarini and Bullo, Francesco},
  journal={IEEE Transactions on Automatic Control},
  volume={60},
  number={12},
  pages={3156--3167},
  year={2015},
  publisher={IEEE}
}

@article{hunter1,
  title={Generalized inverses and their application to applied probability problems},
  author={Hunter, Jeffrey J},
  journal={Linear Algebra and its Applications},
  volume={45},
  pages={157--198},
  year={1982},
  publisher={Elsevier}
}

@article{hunter2,
  title={Characterizations of generalized inverses associated with {M}arkovian kernels},
  author={Hunter, Jeffrey J},
  journal={Linear Algebra and Its Applications},
  volume={102},
  pages={121--142},
  year={1988},
  publisher={Elsevier}
}

@book{hunter3,
  title={Mathematical techniques of applied probability: Discrete time models: Basic theory},
  author={Hunter, Jeffrey J},
  volume={1},
  year={2014},
  publisher={Academic Press}
}

@Book{kemeny1976finite,
  author    = {Kemeny, John G and Snell, J Laurie},
  title     = {Finite {M}arkov {C}hains {W}ith a {N}ew {A}ppendix "{G}eneralization of a {F}undamental {M}atrix"},
  publisher = {Springer},
  year      = {1976},
}

@article{fortunato2010community,
  title={Community detection in graphs},
  author={Fortunato, Santo},
  journal={Physics Reports},
  volume={486},
  number={3-5},
  pages={75--174},
  year={2010},
  publisher={Elsevier}
}

@article{berkhout2019analysis,
  title={Analysis of {M}arkov influence graphs},
  author={Berkhout, Joost and Heidergott, Bernd F},
  journal={Operations Research},
  volume={67},
  number={3},
  pages={892--904},
  year={2019},
  publisher={INFORMS}
}

@book{kemeny1969finite,
  title={Finite {M}arkov chains},
  author={Kemeny, John G and Snell, J Laurie and others},
  volume={26},
  year={1969},
  publisher={van Nostrand Princeton, NJ}
}

@article{kirkland2010fastest,
  title={Fastest expected time to mixing for a {M}arkov chain on a directed graph},
  author={Kirkland, Steve},
  journal={Linear Algebra and its Applications},
  volume={433},
  number={11-12},
  pages={1988--1996},
  year={2010},
  publisher={Elsevier}
}

@article{brin1998anatomy,
  title={The anatomy of a large-scale hypertextual web search engine},
  author={Brin, Sergey and Page, Lawrence},
  journal={Computer networks and ISDN systems},
  volume={30},
  number={1-7},
  pages={107--117},
  year={1998},
  publisher={Elsevier}
}

@article{guo2009mean,
  title={Mean-variance criteria for finite continuous-time {M}arkov decision processes},
  author={Guo, Xianping and Song, Xinyuan},
  journal={IEEE Transactions on Automatic Control},
  volume={54},
  number={9},
  pages={2151--2157},
  year={2009},
  publisher={IEEE}
}

@article{xia2020risk,
  title={Risk-sensitive {M}arkov decision processes with combined metrics of mean and variance},
  author={Xia, Li},
  journal={Production and Operations Management},
  volume={29},
  number={12},
  pages={2808--2827},
  year={2020},
  publisher={SAGE Publications Sage CA: Los Angeles, CA}
}

@article{Markowitz1952,
 ISSN = {00221082, 15406261},
 author = {Harry Markowitz},
 journal = {The Journal of Finance},
 number = {1},
 pages = {77--91},
 publisher = {[American Finance Association, Wiley]},
 title = {Portfolio Selection},
 urldate = {2026-01-19},
 volume = {7},
 year = {1952}
}

@book{magnus2019matrix,
  title={Matrix differential calculus with applications in statistics and econometrics},
  author={Magnus, Jan R and Neudecker, Heinz},
  year={2019},
  publisher={John Wiley \& Sons}
}

@article{franssen2025first,
  title={A First-Order Gradient Approach for the Connectivity Optimization of {M}arkov Chains},
  author={Franssen, Christian PC and Zocca, Alessandro and Heidergott, Bernd F},
  journal={IEEE Transactions on Automatic Control},
  year={2025},
  publisher={IEEE}
}

@article{Bienstock_Verma_2010, title={The $N-k$ Problem in Power Grids: New Models, Formulations, and Numerical Experiments}, volume={20}, number={5}, journal={SIAM Journal on Optimization}, author={Bienstock, Daniel and Verma, Abhinav}, year={2010}, month=jan, pages={2352–2380}}

@article{ellens_2011,
author={W. Ellens and F.M. Spieksma and P. Van Mieghem and A. Jamakovic and R.E. Kooij},
title={Effective graph resistance},
journal={Linear Algebra and its Applications},
volume={435},
number={10},
year={2011},
pages={2491-2506}
}

@article{duan2021stochastic,
  title={Stochastic strategies for robotic surveillance as {S}tackelberg games},
  author={Duan, Xiaoming and Paccagnan, Dario and Bullo, Francesco},
  journal={IEEE Transactions on Control of Network Systems},
  volume={8},
  number={2},
  pages={769--780},
  year={2021},
  publisher={IEEE}
}

@article{george2018markov,
  title={Markov chains with maximum entropy for robotic surveillance},
  author={George, Mishel and Jafarpour, Saber and Bullo, Francesco},
  journal={IEEE Transactions on Automatic Control},
  volume={64},
  number={4},
  pages={1566--1580},
  year={2018},
  publisher={IEEE}
}

@article{duan2019markov,
  title={Markov chains with maximum return time entropy for robotic surveillance},
  author={Duan, Xiaoming and George, Mishel and Bullo, Francesco},
  journal={IEEE Transactions on Automatic Control},
  volume={65},
  number={1},
  pages={72--86},
  year={2019},
  publisher={IEEE}
}

@article{duan2021markov,
  title={Markov chain--based stochastic strategies for robotic surveillance},
  author={Duan, Xiaoming and Bullo, Francesco},
  journal={Annual Review of Control, Robotics, and Autonomous Systems},
  volume={4},
  number={1},
  pages={243--264},
  year={2021},
  publisher={Annual Reviews}
}

@article{altafini2023edge,
  title={An edge centrality measure based on the {K}emeny constant},
  author={Altafini, Diego and Bini, Dario A and Cutini, Valerio and Meini, Beatrice and Poloni, Federico},
  journal={SIAM Journal on Matrix Analysis and Applications},
  volume={44},
  number={2},
  pages={648--669},
  year={2023},
  publisher={SIAM}
}

@article{bini2024kemeny,
  title={On {K}emeny's constant and stochastic complement},
  author={Bini, Dario Andrea and Durastante, Fabio and Kim, Sooyeong and Meini, Beatrice},
  journal={Linear Algebra and its Applications},
  volume={703},
  pages={137--162},
  year={2024},
  publisher={Elsevier}
}

@article{Albert_Jeong_Barabasi_2000,
title={Error and attack tolerance of complex networks}, 
volume={406}, 
number={6794}, 
journal={Nature}, 
author={Albert, Réka and Jeong, Hawoong and Barabási, Albert-László}, year={2000}, month={july}, pages={378–382}, language={en} }

@article{Jenelius_Petersen_Mattsson_2006, title={Importance and exposure in road network vulnerability analysis}, volume={40}, number={7}, journal={Transportation Research Part A: Policy and Practice}, author={Jenelius, Erik and Petersen, Tom and Mattsson, Lars-Göran}, year={2006}, month=aug, pages={537–560} }

@article{Crisostomi_Kirkland_Shorten_2011, title={A {G}oogle-like model of road network dynamics and its application to regulation and control}, volume={84}, number={3}, journal={International Journal of Control}, author={Crisostomi, E. and Kirkland, S. and Shorten, R.}, year={2011}, month=mar, pages={633–651}, language={en} }

@book{Penrose_2003, title={Random Geometric Graphs}, publisher={Oxford University Press}, author={Penrose, Mathew}, year={2003}, month=may }


\newpage
\begin{appendices}
\section{Proof of Theorem~\ref{th: second momemt}}\label{sec:prooftheorem2}

We begin by computing the second moment of the first return times. 
Next, we derive an implicit equation for $\mM^{(2)}$.
By computation,
\begin{dmath*}
\tau^2 ( i , j )  =  \mP ( i , j )  {\cal W}^{(2)} ( i ,j ) + \sum_{k \not = j }\mP ( i , k ) \Big  ( {\cal W}  ( i , k ) + \tau ( k ,j )   \Big )^2   =   \mP ( i , j )   {\cal W}^{2} ( i ,j )  + \sum_{k \not = j }\mP ( i , k ) \Big (  {\cal W}^2 ( k ,j  )   +  2  {\cal W} ( i , k )  \tau ( k ,j )   +  \tau^2 ( k , j )   \Big ) .
\end{dmath*}
Taking expected values gives
\begin{dmath*}
 \mM^{(2)} ( i , j ) 
 = \mP ( i , j )   \mW^{(2)} ( i ,j )  + \sum_{k \not = j }\mP ( i , k ) \Big (  \mW^{(2)} ( k ,j  )   +  2  \mW ( i , k )  \mM( k ,j )   +  \mM^{(2)} ( k , j )  \Big )
 = \sum_{k  } \mP ( i , k )  \mW^2 ( i ,k ) +\sum_{k \not = j }\mP ( i , k )  \mM^{(2)} (k ,j )  +  2 \sum_{k \not = j }\mP ( i , k )   \mW ( i ,k )  \mM ( k ,j ) .
\end{dmath*}
We note that we use $ \mathbb{E} [{\cal W} ( i , k )  \tau ( k ,j ) ] =\mW ( i ,k )  \mM ( k ,j ) $ in the above computation, which is allowed due to assumption (A$^*$).
Indeed, it might happen that the random walker after the transition $ (i , k )$ traverses $ ( i , k ) $ again, which is captured in $ \tau ( k ,j )$. 
However, due to the assumption of independence, such weights are independent of the weight $ {\cal W} ( i , k ) $ appearing in the above equation.

In matrix notation,
\begin{dmath*}
\mM^{(2)} =  \big (\mP \circ \mW^{(2)}  \big ) {\bf 1} {\bf 1}^\top  + \mP ( \mM^{(2)} - [ \mM^{(2)}]_{ \rm{dg}} )
+ 2
\sum_{k \not = j } (\mP   \circ  \mW ) ( i , k )   \mM ( k ,j )
= 
\big ( \mP \circ \mW^{ (2)}  \big ) {\bf 1} {\bf 1}^\top + \mP \big ( \mM^{(2)} - [ \mM^{(2)}]_{ \rm{dg}}  \big )
+ 2
(\mP   \circ  \mW ) \big (   \mM - [ \mM]_{ \rm{dg}}  \big )\, ,
\end{dmath*}
resulting in
\begin{dmath*}
(\mI - \mP)\mM^{(2)} = \big ( \mP \circ \mW^{ (2)}\big ) {\bf 1} {\bf 1}^\top - \mP \big [ \mM^{(2)}\big ]_{ \rm{dg}}
+ 2
(\mP   \circ  \mW ) \big (   \mM - [ \mM]_{ \rm{dg}}  \big )\, .
\end{dmath*}

The next step of the proof is to solve the diagonal of 
$ \mM^{ (2)}$.
Multiplying the above by $ \pi $ gives
\begin{dmath*}
\pi \, \mM^{(2)}  = 
\pi 
\big ( \mP \circ \mW^{ (2)} \big ) {\bf 1} {\bf 1}^\top + \pi  \mP \big ( \mM^{(2)} - [ \mM^{(2)}]_{ \rm{dg}}  \big )
+ 2
\pi  (\mP   \circ  \mW ) \big (   \mM - [ \mM]_{ \rm{dg}}  \big ) 
\end{dmath*}
This gives
\begin{dmath}\label{diagonal sec moment}
\pi [ \mM^{(2)}]_{ \rm{dg}}  
=  \pi ( \mP \circ \mW^{(2)} ) {\bf 1} {\bf 1}^\top 
+
2 \pi  (\mP   \circ  \mW ) \big (   \mM - [ \mM]_{ \rm{dg}}  \big )  
\end{dmath}
and we obtain for the diagonal elements
\begin{dmath*}
\mM^{(2)} ( i , i )   
= \frac{1}{\pi ( i ) }\Big ( 
\pi ( \mP \circ \mW^{(2)} ) {\bf 1} {\bf 1}^\top 
+
2 \pi  (\mP   \circ  \mW ) \big (   \mM - [ \mM]_{ \rm{dg}}  \big )  
\Big )  ( i ) .
\end{dmath*}
Noting that the $ i$-the entry of 
$  \pi ( \mP \circ \mW^{(2)} ) {\bf 1} {\bf 1}^\top $ is equal
to $ \pi ( \mP \circ \mW^{(2)} ) {\bf 1} $ and independent of $ i$, yields the claimed expression for the diagonal elements.

Again, we will rely on the $\mZ$ as the generalized inverse of $(\mI - \mP)$ to find a closed-form formula for $\mM^{(2)}$. First, we show that $\mZ$ is again consistent, that is, it solves,
\begin{dmath*}
    (\mI-(\mI-\mP)\mZ)\left((\mP\circ \mW^{(2)}) {\bf 1} {\bf 1}^\top
    +
    2(\mP \circ \mW)(\mM - [\mM]_{\rm{dg}} )
    -
    \mP[\mM^{(2)}]_{\rm{dg}}\right) = 0\, .
\end{dmath*}
Recall that,
\[
 (\mI-(\mI-\mP)\mZ) = \Pi\, ,
\]
and thus the left-hand side can be simplified to,
\begin{dmath*}
    \Pi\left((\mP\circ \mW^{(2)}) {\bf 1} {\bf 1}^\top
    +
    2(\mP \circ \mW)(\mM - [\mM]_{\rm{dg}} )
    -
    \mP[\mM^{(2)}]_{\rm{dg}}\right)\nonumber
    = {\bf 1} \Big(\vpi( \mP\circ \mW^{(2)}) {\bf 1} {\bf 1}^\top
    +
    2\vpi(\mP \circ \mW)(\mM - [\mM]_{\rm{dg}} )
    -
    \vpi\mP[\mM^{(2)}]_{\rm{dg}}\Big)\nonumber
    = {\bf 1} \Big(\vpi( \mP\circ \mW^{(2)}) {\bf 1} {\bf 1}^\top
    +
    2\vpi(\mP \circ \mW)(\mM - [\mM]_{\rm{dg}} )
    -
    \vpi[\mM^{(2)}]_{\rm{dg}}\Big)\, . \nonumber
\end{dmath*}
By directly substituting \eqref{diagonal sec moment} into this equation, we see that all terms cancel out, resulting in ${\bf 1} 0 = 0$. Thus $\mZ$ is consistent once more, and, according to the generalized inverse theory provided earlier, a solution to the system of equations is given by
\begin{dmath*}
\mM^{(2)} = \mZ \Big( (\mP \circ \mW^{(2)} ) {\bf 1} {\bf 1}^\top
+
2(\mP \circ \mW)\big(\mM - [\mM]_{\rm{dg}} \big)
-
\mP [\mM^{(2)}]_{\rm{dg}}
\Big) + \Pi\tilde{\mU}\, ,
\end{dmath*}
where $\tilde{\mU}$ is again an arbitrary matrix, distinct from the earlier used $\mU$. Since $\Pi$ has equal rows, so has $\Pi \tilde{\mU}$ and we can write
\begin{dmath}\label{eq: still have to solve for tilde u}
\mM^{(2)} = \mZ \Big( (\mP \circ \mW^{(2)} ) {\bf 1} {\bf 1}^\top
+
2(\mP \circ \mW)\big(\mM - [\mM]_{\rm{dg}} \big)
-
\mP [\mM^{(2)}]_{\rm{dg}}
\Big) + {\bf 1} \tilde{\vu}^\top \, ,
\end{dmath}
for some vector $\tilde{\vu}$.

We can solve $\tilde{\vu}$ by observing the diagonal entries of \eqref{eq: still have to solve for tilde u}
\begin{dmath*}
{\rm{dg}}(\tilde{\vu}) = [\mM^{(2)}]_{\rm{dg}} - \Big[\mZ \Big( (\mP \circ \mW^{(2)} ) {\bf 1} {\bf 1}^\top
+
2(\mP \circ \mW)\big(\mM - [\mM]_{\rm{dg}} \big)
-
\mP [\mW^{(2)}]_{\rm{dg}}
\Big)\Big]_{\rm{dg}}\nonumber
=\big(\mI + [\mZ\mP]_{\rm{dg}} \big)[\mM^{(2)}]_{\rm{dg}}
-
\big[ \mZ (\mP \circ \mW^{(2)} ) {\bf 1} {\bf 1}^\top \big]_{\rm{dg}}
-
2 \big[ \mZ (\mP \circ \mW)(\mM - [\mM]_{\rm{dg}})\big]_{\rm{dg}}
\, . \nonumber
\end{dmath*}
From simple algebra we know that ${\bf 1} \tilde{\vu} = {\bf 1} {\bf 1}^\top {\rm{dg}}(\tilde{\vu})$, which we can apply directly to \eqref{eq: still have to solve for tilde u},
\begin{dmath}
\mM^{(2)} = \mZ \Big( (\mP \circ \mW^{(2)} ) {\bf 1} {\bf 1}^\top
+
2(\mP \circ \mW)\big(\mM - [\mM]_{\rm{dg}} \big)
-
\mP [\mM^{(2)}]_{\rm{dg}}
\Big)\nonumber 
+ {\bf 1} {\bf 1}^\top \Big( \big(\mI + [\mZ\mP]_{\rm{dg}} \big)[\mM^{(2)}]_{\rm{dg}}
-
\big[ \mZ (\mP \circ \mW^{(2)} ) {\bf 1} {\bf 1}^\top \big]_{\rm{dg}}\Big)\nonumber
- 2 \, {\bf 1} {\bf 1}^\top \big[ \mZ (\mP \circ \mW)(\mM - [\mM]_{\rm{dg}})\big]_{\rm{dg}}\nonumber
= \mZ (\mP \circ \mW^{(2)}) {\bf 1} {\bf 1}^\top
- {\bf 1} {\bf 1}^\top \big[ \mZ(\mP \circ \mW^{(2)}) {\bf 1} {\bf 1}^\top\big]_{\rm{dg}}\nonumber
+ 2 \Big(
\mZ ( \mP \circ \mW)(\mM - [\mM]_{\rm{dg}})
-
{\bf 1} {\bf 1}^\top \big[ \mZ ( \mP \circ \mW) (\mM - [\mM]_{\rm{dg}})\big]_{\rm{dg}}
\Big)\nonumber
+ \Big(
{\bf 1} {\bf 1}^\top - \mZ\mP + {\bf 1} {\bf 1}^\top [\mZ\mP]_{\rm{dg}}
\Big) [\mM^{(2)}]_{\rm{dg}}\, . \label{insert equality here}
\end{dmath}
Recall that
\[
{\bf 1} {\bf 1}^\top - \mZ \mP  + [ \mZ \mP ]_{\rm{dg}}
= \mI - \mZ + {\bf 1} {\bf 1}^\top [ \mZ ]_{\rm{dg}}  .
\]
Inserting this equality into \eqref{insert equality here}, we arrive at the final representation:
\begin{dmath*}
\mM^{(2)} = 
\mZ ( \mP \circ \mW^{(2)} ) {\bf 1} {\bf 1}^\top
-
{\bf 1} {\bf 1}^\top \Big[ \mZ ( \mP \circ \mW^{(2)}) {\bf 1} {\bf 1}^\top \Big] \nonumber \\
+
2 \Big( \mZ ( \mP \circ \mW) \big( \mM - [ \mM]_{ \rm{dg}}\big)
-
{\bf 1} {\bf 1}^\top \Big[ \mZ ( \mP \circ \mW) \big( \mM - [ \mM]_{ \rm{dg}}\big) \Big]_{\rm{dg}} \Big) \nonumber \\
+
\big( \mI - \mZ + {\bf 1} {\bf 1}^\top [\mZ]_{\rm{dg}} \big) \big[ \mM^{(2)}\big]_{\rm{dg}}\, . \nonumber
\end{dmath*}
This concludes the proof.

\section{Proof of Lemma~\ref{lem:derivative_pi_w}\label{proof of lemma 4}}

From equation \eqref{pi_W}, we have:
\[
\vpi_{ \cal W}(i) = \frac{\vpi(i)\bar{U}(i)}{Y}, \quad \text{where } Y = \sum_{j \in \sV} \vpi(j)\bar{U}(j)
\]
Taking the derivative of the numerator with respect to $\mW(l,k)$:
\begin{align*}
    \frac{\partial[\vpi(i)\bar{U}(i)]}{\partial \mW(l,k)} &= \vpi(i)\frac{\partial \bar{U}(i)}{\partial \mW(l,k)}\\
&= \vpi(i)\frac{\partial}{\partial \mW(l,k)}\sum_{j \in \sV}\mP(i,j)\mW(i,j)\\
&= \begin{cases}
\vpi(i)\mP(i,k) & \text{if } l = i ,\\
0 & \text{otherwise.}
\end{cases}
\end{align*}
Taking the derivative of the denominator:
\begin{align*}
    \frac{\partial Y}{\partial \mW(l,k)} &= \frac{\partial}{\partial \mW(l,k)}\sum_{j \in \sV}\vpi(j)\bar{U}(j)\\
&= \vpi(l)\frac{\partial \bar{U}(l)}{\partial \mW(l,k)}
= \vpi(l)\mP(l,k) .
\end{align*}
Applying the quotient rule:
\begin{dmath*}
    \frac{\partial \vpi_{ \cal W}(i)}{\partial \mW(l,k)} = \frac{1}{Y^2} \left[ \frac{\partial[\vpi(i)\bar{U}(i)]}{\partial \mW(l,k)} \cdot Y
     - \vpi(i)\bar{U}(i) \cdot \frac{\partial Y}{\partial \mW(l,k)} \right] .
\end{dmath*}
For $l = i$:
\begin{dmath*}
    \frac{\partial \vpi_{ \cal W}(i)}{\partial \mW(i,k)} = \frac{1}{Y^2}\left[\vpi(i)\mP(i,k) \cdot Y - \vpi(i)\bar{U}(i) \cdot \vpi(i)\mP(i,k)\right]
= \frac{\vpi(i)\mP(i,k)}{Y^2}\left[Y - \vpi(i)\bar{U}(i)\right] .
\end{dmath*}
For $l \neq i$:
\begin{dmath*}
    \frac{\partial \vpi_{ \cal W}(i)}{\partial \mW(l,k)} = \frac{1}{Y^2}\left[0 \cdot Y - \vpi(i)\bar{U}(i) \cdot \vpi(l)\mP(l,k)\right]
= -\frac{\vpi(i)\bar{U}(i)\vpi(l)\mP(l,k)}{Y^2}.
\end{dmath*}
This concludes the proof.

\section{Derivative of Second Moment, Mean Weights} \label{derivative_second_moment}

\begin{lemma}\label{lem:derivative_second_moment}
Provided that (A) and (A$^*$) hold, the partial derivative of the second moment matrix $\mM^{(2)}$ with respect to $\mW(l,k)$ is given by:
\begin{dmath*}
\frac{\partial \mM^{(2)}}{\partial \mW(l,k)} = 2\mW(l,k)\mP(l,k)\left[\mZ\mathbf{e}_l\mathbf{1}^\top - \mathbf{1}\mathbf{1}^\top [\mZ\mathbf{e}_l\mathbf{1}^\top]_{\rm{dg}}\right]
+ 2\mP(l,k)\left[\mZ\mathbf{e}_l\mathbf{e}_k^\top(\mM - [\mM]_{\rm{dg}}) - \mathbf{1}\mathbf{1}^\top [\mZ\mathbf{e}_l\mathbf{e}_k^\top(\mM - [\mM]_{\rm{dg}})]_{\rm{dg}}\right]  
+ 2\left[\mZ(\mP \circ \mW)\left(\frac{\partial \mM}{\partial \mW(l,k)} - \left[\frac{\partial \mM}{\partial \mW(l,k)}\right]_{\rm{dg}}\right)
- \mathbf{1}\mathbf{1}^\top \left[\mZ(\mP \circ \mW)\left(\frac{\partial \mM}{\partial \mW(l,k)} - \left[\frac{\partial \mM}{\partial \mW(l,k)}\right]_{\rm{dg}}\right)\right]_{\rm{dg}}\right]
+ (\mI - \mZ + \mathbf{1}\mathbf{1}^\top [\mZ]_{\rm{dg}})\frac{\partial [\mM^{(2)}]_{\rm{dg}}}{\partial \mW(l,k)} ,
\end{dmath*}
where $\frac{\partial \mM}{\partial \mW(l,k)}$ is given by Lemma~\ref{lem:derivative_mfpt}, and the diagonal derivative $\frac{\partial [\mM^{(2)}]_{\rm{dg}}}{\partial \mW(l,k)}$ has the explicit form
\begin{dmath*}
\frac{\partial \mM^{(2)}(i,i)}{\partial \mW(l,k)} = \frac{2\mW(l,k)\vpi(l)\mP(l,k)}{\vpi(i)}+ \frac{2\vpi(l)\mP(l,k)[\mM - [\mM]_{\rm{dg}}](i,k)}{\vpi(i)} + \frac{2\left(\vpi(\mP \circ \mW)\left(\frac{\partial \mM}{\partial \mW(l,k)} - \left[\frac{\partial \mM}{\partial \mW(l,k)}\right]_{\rm{dg}}\right)\right)(i)}{\vpi(i)}.
\end{dmath*}
\end{lemma}

\noindent \textbf{Proof:}

We begin by noting that for deterministic weights, $\frac{\partial \mW^{(2)}(i,j)}{\partial \mW(l,k)} = 2\mW(l,k)\delta_{il}\delta_{jk}$, and thus 
\begin{dmath*}
    \frac{\partial(\mP \circ \mW^{(2)})}{\partial \mW(l,k)} = 2\mW(l,k)\mP(l,k)\mathbf{e}_l\mathbf{e}_k^\top.
\end{dmath*}
For the proof of the first part of the lemma, we start by noting that Theorem~ \ref{th: second momemt} yields 
\begin{dmath*}
\mM^{(2)} = \mZ(\mP \circ \mW^{(2)})\mathbf{1}\mathbf{1}^\top - \mathbf{1}\mathbf{1}^\top [\mZ(\mP \circ \mW^{(2)})\mathbf{1}\mathbf{1}^\top]_{\rm{\dg}} 
+ 2[\mZ(\mP \circ \mW)(\mM - [\mM]_{\rm{dg}}) - \mathbf{1}\mathbf{1}^\top [\mZ(\mP \circ \mW)(\mM - [\mM]_{\rm{dg}})]_{\rm{dg}}]
+ (\mI - \mZ + \mathbf{1}\mathbf{1}^\top [\mZ]_{\rm{dg}})[\mM^{(2)}]_{\rm{dg}}.
\end{dmath*}
In the following, we take the derivative of each term:

\noindent
First term:
\begin{dmath*}
\frac{\partial}{\partial \mW(l,k)}[\mZ(\mP \circ \mW^{(2)})\mathbf{1}\mathbf{1}^\top] = \mZ\frac{\partial(\mP \circ \mW^{(2)})}{\partial \mW(l,k)}\mathbf{1}\mathbf{1}^\top
= \mZ \cdot 2\mW(l,k)\mP(l,k)\mathbf{e}_l\mathbf{e}_k^\top\mathbf{1}\mathbf{1}^\top
= 2\mW(l,k)\mP(l,k)\mZ\mathbf{e}_l\mathbf{1}^\top ,
\end{dmath*}
Second term:
\begin{dmath*}
\frac{\partial}{\partial \mW(l,k)}[\mathbf{1}\mathbf{1}^\top [\mZ(\mP \circ \mW^{(2)})\mathbf{1}\mathbf{1}^\top]_{\rm{\dg}}] = \mathbf{1}\mathbf{1}^\top \left[\frac{\partial}{\partial \mW(l,k)}[\mZ(\mP \circ \mW^{(2)})\mathbf{1}\mathbf{1}^\top]\right]_{\rm{dg}} 
= 2\mW(l,k)\mP(l,k)\mathbf{1}\mathbf{1}^\top [\mZ\mathbf{e}_l\mathbf{1}^\top]_{\rm{dg}},
\end{dmath*}
Third term (with factor 2):
\begin{dmath*}
\frac{\partial}{\partial \mW(l,k)}[\mZ(\mP \circ \mW)(\mM - [\mM]_{\rm{dg}})] 
= \mZ\frac{\partial(\mP \circ \mW)}{\partial \mW(l,k)}(\mM - [\mM]_{\rm{dg}}) + \mZ(\mP \circ \mW)\frac{\partial(\mM - [\mM]_{\rm{dg}})}{\partial \mW(l,k)} 
= \mZ\mP(l,k)\mathbf{e}_l\mathbf{e}_k^\top(\mM - [\mM]_{\rm{dg}}) + \mZ(\mP \circ \mW)\left(\frac{\partial \mM}{\partial \mW(l,k)} - \left[\frac{\partial \mM}{\partial \mW(l,k)}\right]_{\rm{dg}}\right) ,
\end{dmath*}
Fourth term (inside the factor 2):
\begin{dmath*}
    \frac{\partial}{\partial \mW(l,k)}[\mathbf{1}\mathbf{1}^\top [\mZ(\mP \circ \mW)(\mM - [\mM]_{\rm{dg}})]_{\rm{dg}}] = \mathbf{1}\mathbf{1}^\top \left[\frac{\partial}{\partial \mW(l,k)}[\mZ(\mP \circ \mW)(\mM - [\mM]_{\rm{dg}})]\right]_{\rm{dg}} ,
\end{dmath*}
Fifth term:
\begin{dmath*}
    \frac{\partial}{\partial \mW(l,k)}[(\mI - \mZ + \mathbf{1}\mathbf{1}^\top [\mZ]_{\rm{dg}})[\mM^{(2)}]_{\rm{dg}}] = (\mI - \mZ + \mathbf{1}\mathbf{1}^\top [\mZ]_{\rm{dg}})\frac{\partial [\mM^{(2)}]_{\rm{dg}}}{\partial \mW(l,k)}.
\end{dmath*}
Combining all terms yields the first result.

We now turn to the explicit solution for the diagonal derivatives.
From Theorem~\ref{th: second momemt}, the diagonal elements of $\mM^{(2)}$ satisfy:
\begin{dmath*}
    \mM^{(2)}(i,i) = \frac{\vpi(\mP \circ \mW^{(2)})\mathbf{1} + 2[\vpi(\mP \circ \mW)(\mM - [\mM]_{\rm{dg}})]_i}{\vpi(i)} .
\end{dmath*}
Taking the derivative with respect to $\mW(l,k)$:
\begin{dmath*}
\frac{\partial \mM^{(2)}(i,i)}{\partial \mW(l,k)} = \frac{1}{\vpi(i)}\left[\frac{\partial[\vpi(\mP \circ \mW^{(2)})\mathbf{1}]}{\partial \mW(l,k)} + 2\frac{\partial[\vpi(\mP \circ \mW)(\mM - [\mM]_{\rm{dg}})]_i}{\partial \mW(l,k)}\right] .
\end{dmath*}
For the first term, using $\frac{\partial(\mP \circ \mW^{(2)})}{\partial \mW(l,k)} = 2\mW(l,k)\mP(l,k)\mathbf{e}_l\mathbf{e}_k^\top$:
\begin{dmath*}
    \frac{\partial[\vpi(\mP \circ \mW^{(2)})\mathbf{1}]}{\partial \mW(l,k)} = \vpi \cdot 2\mW(l,k)\mP(l,k)\mathbf{e}_l\mathbf{e}_k^\top\mathbf{1} = 2\mW(l,k)\vpi(l)\mP(l,k) .
\end{dmath*}
For the second term, using the product rule:
\begin{dmath*}
\frac{\partial[\vpi(\mP \circ \mW)(\mM - [\mM]_{\rm{dg}})]_i}{\partial \mW(l,k)} = \vpi(l)\mP(l,k)[\mM - [\mM]_{\rm{dg}}](i,k) + \left[\vpi(\mP \circ \mW)\frac{\partial(\mM - [\mM]_{\rm{dg}})}{\partial \mW(l,k)}\right]_i .
\end{dmath*}
Putting the above terms together provides an explicit formula for each diagonal element.
This concludes the proof.

\section{Gradients with respect to the transition matrix} \label{sec:gradp}

For a Markov chain $\mP$, the partial derivative $\frac{\partial \mP}{\partial \mP(l,k)}$ must satisfy the row-sum constraint. Specifically, for each row $i$:
$$\frac{\partial \mP}{\partial \mP(l,k)}\mathbf{1} = \mathbf{0}$$
This ensures that when $\mP(l,k)$ increases by $\delta$, the row sum $\sum_j \mP(l,j)$ remains equal to 1. For a principled way to deal with the row sum constraint in this partial derivative, refer to the study of \cite{franssen2025first}.

\begin{lemma}[Derivative of the Stationary Distribution]\label{lem:derivative_stationary}
Let $\mP$ be the transition matrix of an irreducible Markov chain with stationary distribution $\vpi$ and fundamental matrix $\mZ = (\mI - \mP + \mathbf{1}\vpi)^{-1}$, the derivative of the stationary distribution is:

\begin{equation*}\label{eq:dpi_dP}
\frac{\partial \vpi}{\partial \mP(l,k)} = \vpi \frac{\partial \mP}{\partial \mP(l,k)} \mZ ,
\end{equation*}
where $\frac{\partial \mP}{\partial \mP(l,k)}$ satisfies the row-sum constraint $\frac{\partial \mP}{\partial \mP(l,k)}\mathbf{1} = \mathbf{0}$.
\end{lemma}

\begin{proof}{Proof:}
The stationary distribution satisfies $\vpi\mP = \vpi$ and $\vpi\mathbf{1} = 1$.

Taking the derivative of $\vpi\mP = \vpi$ with respect to $\mP(l,k)$, using the product rule:
$$\frac{\partial \vpi}{\partial \mP(l,k)}\mP + \vpi\frac{\partial \mP}{\partial \mP(l,k)} = \frac{\partial \vpi}{\partial \mP(l,k)}$$

Rearranging:
$$\frac{\partial \vpi}{\partial \mP(l,k)}(\mI - \mP) = \vpi\frac{\partial \mP}{\partial \mP(l,k)}$$

Multiplying both sides by $\mZ$ from the right:
$$\frac{\partial \vpi}{\partial \mP(l,k)}(\mI - \mP)\mZ = \vpi\frac{\partial \mP}{\partial \mP(l,k)}\mZ$$

Using the property $(\mI - \mP)\mZ = \mI - \mathbf{1}\vpi$ (See \cite{kemeny1969finite}, p. 100):
$$\frac{\partial \vpi}{\partial \mP(l,k)}(\mI - \mathbf{1}\vpi) = \vpi\frac{\partial \mP}{\partial \mP(l,k)}\mZ$$

Expanding the left side:
$$\frac{\partial \vpi}{\partial \mP(l,k)} - \frac{\partial \vpi}{\partial \mP(l,k)}\mathbf{1}\vpi = \vpi\frac{\partial \mP}{\partial \mP(l,k)}\mZ$$

Now, from the normalization constraint $\vpi\mathbf{1} = 1$, taking derivatives:
$$\frac{\partial \vpi}{\partial \mP(l,k)}\mathbf{1} = 0$$

This is automatically satisfied because:
$$\frac{\partial \vpi}{\partial \mP(l,k)}\mathbf{1} = \vpi\frac{\partial \mP}{\partial \mP(l,k)}\mZ\mathbf{1} = \vpi\frac{\partial \mP}{\partial \mP(l,k)}\mathbf{1} = \vpi \cdot \mathbf{0} = 0$$

where we used $\mZ\mathbf{1} = \mathbf{1}$ \citep{kemeny1969finite} and $\frac{\partial \mP}{\partial \mP(l,k)}\mathbf{1} = \mathbf{0}$.

Therefore, the first term simplifies:
$$\frac{\partial \vpi}{\partial \mP(l,k)} - 0 = \vpi\frac{\partial \mP}{\partial \mP(l,k)}\mZ .$$
Hence:
$$\frac{\partial \vpi}{\partial \mP(l,k)} = \vpi\frac{\partial \mP}{\partial \mP(l,k)}\mZ .$$
\end{proof}

\begin{lemma}[Derivative of the Fundamental Matrix]\label{lem:derivative_fundamental}
The derivative of the fundamental matrix $\mZ = (\mI - \mP + \mathbf{1}\vpi)^{-1}$ with respect to $\mP(l,k)$ is:
\begin{equation*}\label{eq:dZ_dP}
\frac{\partial \mZ}{\partial \mP(l,k)} = \mZ\left(\frac{\partial \mP}{\partial \mP(l,k)} - \mathbf{1}\frac{\partial \vpi}{\partial \mP(l,k)}\right)\mZ,
\end{equation*}
where $\frac{\partial \mP}{\partial \mP(l,k)}$ is the feasible perturbation that satisfies the row sum constraints and $\frac{\partial \vpi}{\partial \mP(l,k)}$ is given by Lemma~\ref{lem:derivative_stationary}.
\end{lemma}

\begin{proof}{Proof:}
We use the matrix inverse derivative formula \cite{magnus2019matrix}.
For an invertible matrix $\mX(\theta)$ that depends on parameter $\theta$:
$$\frac{\partial \mX^{-1}}{\partial \theta} = -\mX^{-1}\frac{\partial \mX}{\partial \theta}\mX^{-1} .$$
Applied to $\mZ = (\mI - \mP + \mathbf{1}\vpi)^{-1}$:
$$\frac{\partial \mZ}{\partial \mP(l,k)} = -\mZ\frac{\partial(\mI - \mP + \mathbf{1}\vpi)}{\partial \mP(l,k)}\mZ .$$
Computing the derivative inside:
\begin{align*}
\frac{\partial(\mI - \mP + \mathbf{1}\vpi)}{\partial \mP(l,k)} 
&= \frac{\partial \mI}{\partial \mP(l,k)} - \frac{\partial \mP}{\partial \mP(l,k)} + \frac{\partial(\mathbf{1}\vpi)}{\partial \mP(l,k)} \\
&= 0 - \frac{\partial \mP}{\partial \mP(l,k)} + \mathbf{1}\frac{\partial \vpi}{\partial \mP(l,k)} \\
&= -\frac{\partial \mP}{\partial \mP(l,k)} + \mathbf{1}\frac{\partial \vpi}{\partial \mP(l,k)},
\end{align*}
and substituting back
\begin{align*}
\frac{\partial \mZ}{\partial \mP(l,k)} 
&= -\mZ\left(-\frac{\partial \mP}{\partial \mP(l,k)} + \mathbf{1}\frac{\partial \vpi}{\partial \mP(l,k)}\right)\mZ \\
&= \mZ\left(\frac{\partial \mP}{\partial \mP(l,k)} - \mathbf{1}\frac{\partial \vpi}{\partial \mP(l,k)}\right)\mZ .
\end{align*}
\end{proof}

\begin{theorem}[Derivative of Mean first passage time Weights with respect to $ \mP$]\label{thm:derivative_MT_wrt_P}
Provided that (A) holds, the partial derivative of the mean first passage weight with respect to the transition probability $ \mP ( l , k ) $ is given by
\begin{dmath*}
\frac{\partial \mM}{\partial \mP(l,k)} 
= \Bigg[
\frac{\partial \mZ}{\partial \mP(l,k)} 
\Big(
(\mP \circ \mW)\mathbf{1}\vpi 
- \mathbf{1}\mathbf{1}^\top [(\mP \circ \mW)\mathbf{1}\vpi]_{\rm{dg}} 
+ C(\mI - \mZ + \mathbf{1}\mathbf{1}^\top [\mZ]_{\rm{dg}})
\Big)
+ \mZ\Big(
\frac{\partial(\mP \circ \mW)}{\partial \mP(l,k)}\mathbf{1}\vpi 
+ (\mP \circ \mW)\mathbf{1}\frac{\partial \vpi}{\partial \mP(l,k)} 
- \mathbf{1}\mathbf{1}^\top
\Big[
\frac{\partial(\mP \circ \mW)}{\partial \mP(l,k)}\mathbf{1}\vpi 
+ (\mP \circ \mW)\mathbf{1}\frac{\partial \vpi}{\partial \mP(l,k)}
\Big]_{\rm{dg}}
\Big)
+ \frac{\partial C}{\partial \mP(l,k)}(\mI - \mZ + \mathbf{1}\mathbf{1}^\top [\mZ]_{\rm{dg}})
- C\frac{\partial \mZ}{\partial \mP(l,k)} 
+ C\mathbf{1}\mathbf{1}^\top 
\Big[\frac{\partial \mZ}{\partial \mP(l,k)}\Big]_{\rm{dg}}
\Bigg]\Xi^{-1}
- 
\Big[
\mZ(\mP \circ \mW)\mathbf{1}\vpi 
- \mathbf{1}\mathbf{1}^\top [\mZ(\mP \circ \mW)\mathbf{1}\vpi]_{\rm{dg}} 
+ C(\mI - \mZ + \mathbf{1}\mathbf{1}^\top [\mZ]_{\rm{dg}})
\Big]\Xi^{-1}
\frac{\partial \vpi}{\partial \mP(l,k)}\Xi^{-1} .
\end{dmath*}
\end{theorem}

\begin{proof}{Proof:}
Write $\mM = \mF[\mZ, \vpi, \mP, \mW] \cdot \Xi^{-1}[\vpi]$, where $\mF$ denotes the bracketed expression. The formula for $\mM$ depends on $\mP$ both directly through $(\mP \circ \mW)$ and indirectly through $\mZ$, $\vpi$, $C = \vpi(\mP \circ \mW)\mathbf{1}$, and $\Xi^{-1}$. Applying the product rule:
$$\frac{\partial \mM}{\partial \mP(l,k)} = \frac{\partial \mF}{\partial \mP(l,k)} \Xi^{-1} 
+ \mF \frac{\partial \Xi^{-1}}{\partial \mP(l,k)} .$$
For the second term, since $\Xi^{-1} = \text{diag}(1/\vpi(1), \ldots, 1/\vpi(n))$:
$$\frac{\partial \Xi^{-1}}{\partial \mP(l,k)} = -\Xi^{-1}\frac{\partial \vpi}{\partial \mP(l,k)}\Xi^{-1} .$$
For $\frac{\partial \mF}{\partial \mP(l,k)}$, we differentiate each of the three terms in $\mF$ using the product rule, propagating derivatives through all $\mP$-dependent quantities. Combining these with the component derivatives from Lemma~\ref{lem:derivative_stationary} and Lemma~\ref{lem:derivative_fundamental} establishes the claim.
\end{proof}

\section{Relation to Taboo Kernel}\label{sec:taboo}

A standard alternative for computing MFPTs is the following.
Let $ _{i} \mP $ denote $ \mP$ with the $ i$th column set to zero. Then, $ _{i} \mP $ is a deficient Markov transition matrix that models the $ \mP$ chain that never enters state $ i$.
By simple algebra, it can be shown that the expected number of steps the random walker takes going from $ i $ back to $ j $ is given by
\[
\mM ( i , j ) = \left ( \sum_{k=0}^{\infty} 
\left ( _{j} \mP \right)^k  \right ) (  i, j )  - 1 .
\]
Since any irreducible finite Markov chain is positive recurrent, the sum on the above right-hand side is finite. It is easily checked that 
\begin{dmath*}
 \left ( \mI - _{j} \! \! \mP \right ) 
\left ( \sum_{k=0}^{\infty} 
\left ( _{j} \mP \right)^k  \right )
=
\sum_{k=0}^{\infty} 
\left ( _{j} \mP \right)^k
=
\left ( \sum_{k=0}^{\infty} 
\left ( _{j} \mP \right)^k  \right ) \left ( \mI - _{j} \! \! \mP  \right )
.
\end{dmath*}
Hence, we can solve it by
\[
\sum_{k=0}^{\infty} \left ( _{j} \mP  \right )^k =
( \mI -  _{j} \! \! \mP )^{-1} .
\]
Unfortunately, this requires inversions of the $n \times n $ matrix to compute all elements of $ \mM$, making this method infeasible to compute $ \mM$.

It is worth noting that \cite{borkar2012hamiltonian} uses this approach and computes the element $ ( i , i ) $ for the VFPT matrix for some fixed node $i$. 
They then show that if this variance approaches zero, $ \mP$ models a Hamiltonian cycle. The argument being that if $ \mP$ is irreducible and VFPT $( i , i ) =0 $ for one specific node $i$, then this deterministic path has to visit all nodes. 
\end{appendices}

\end{document}